\documentclass[a4paper, reqno]{amsart}

\usepackage[margin=1.1in]{geometry}

\usepackage{lmodern}
\usepackage{microtype}

\usepackage{amsmath}        
\usepackage{amsfonts}       
\usepackage{amsthm}         
\usepackage{amssymb}%
\usepackage{bm}             
\usepackage{graphicx}       
\usepackage[dvipsnames]{xcolor}  
\usepackage[colorlinks=true]{hyperref}
\hypersetup{citecolor=Sepia, linkcolor=Sepia, urlcolor=MidnightBlue}
\usepackage[T1]{fontenc} 
\usepackage{mathtools}
\usepackage[shortlabels]{enumitem}
\setlist[enumerate,1]{label=(\alph*), font=\normalfont}
\setlist[enumerate,2]{label=(\roman*), font=\normalfont,, ref=\theenumi(\roman*)}

\usepackage{csquotes}
\usepackage{bbding}

\numberwithin{equation}{section}

\theoremstyle{plain}
\newtheorem{theorem}{Theorem}[section]
\newtheorem{lemma}[theorem]{Lemma}
\newtheorem{proposition}[theorem]{Proposition}
\newtheorem{corollary}[theorem]{Corollary}
\newtheorem{conjecture}[theorem]{Conjecture}
\newtheorem*{KitaokaConj}{Kitaoka's Conjecture}

\theoremstyle{definition}
\newtheorem{definition}[theorem]{Definition}
\newtheorem{remark}[theorem]{Remark}
\newtheorem{example}[theorem]{Example}
\newtheorem*{convention}{Convention}
\newtheorem*{remarkunnum}{Remark}

\DeclareMathOperator{\sign}{sgn}
\DeclareMathOperator{\disc}{disc}
\DeclareMathOperator{\rank}{rank}

\newcommand{\R}{\mathbb{R}}
\newcommand{\Z}{\mathbb{Z}}
\newcommand{\Q}{\mathbb{Q}}

\newcommand{\vct}[1]{\bm{#1}} 
\newcommand{\e}{\vct{e}}
\newcommand{\f}{\vct{f}}
\newcommand{\vu}{\vct{u}}
\newcommand{\vv}{\vct{v}}
\newcommand{\vw}{\vct{w}}
\newcommand{\vx}{\vct{x}}

\newcommand{\ip}{\mathfrak{p}}
\newcommand{\ia}{\mathfrak{a}}

\newcommand{\ms}{\mathfrak{s}} 

\newcommand{\ve}{\varepsilon}
\newcommand{\abs}[1]{\left|#1\right|}

\newcommand{\co}{\mathcal{O}}
\newcommand{\U}{\mathcal{U}}
\newcommand{\UK}[1][K]{\U_{#1}}
\newcommand{\UKPlus}[1][K]{\UK[{#1}]^+}
\newcommand{\UKctv}[1][K]{\UK[{#1}]^2}
\newcommand{\OK}[1][K]{\co_{{\:\!\! #1}}}
\newcommand{\OKPlus}[1][K]{\OK[#1]^+}
\newcommand{\OKctv}[1][K]{\OK[#1]^2}

\newcommand{\qf}[1]{\langle#1\rangle} 
\newcommand{\dual}[1]{#1^{\#}}

\newcommand{\Tr}[2]{\mathrm{Tr}_{#1}\left(#2\right)} 
\newcommand{\norm}[2]{\mathrm{N}_{#1}\left(#2\right)} 
\def\house#1{{%
		\setbox0=\hbox{$#1$}
		\vrule height \dimexpr\ht0+1.4pt width .4pt depth \dp0\relax
		\vrule height \dimexpr\ht0+1.4pt width \dimexpr\wd0+2pt depth \dimexpr-\ht0-1pt\relax
		\llap{$#1$\kern1pt}
		\vrule height \dimexpr\ht0+1.4pt width .4pt depth \dp0\relax
}}

\newcommand{\anglematrix}[3]{\bigl\langle \begin{smallmatrix} #1 & #2 \\ #2 & #3 \end{smallmatrix}\bigr\rangle}

\newcommand{\iks}{x}
\newcommand{\ips}{y}
\newcommand{\iss}{z}
\newcommand{\ifs}{w}
\newcommand{\its}{t}

\title[Non-universality of ternary quadratic forms over fields containing $\sqrt2$]{Non-universality of ternary quadratic forms over fields containing $\mathbf{\sqrt2}$}

\author{Krist\'yna Kramer}
\address{Charles University, Faculty of Mathematics and Physics, Department of Algebra, Sokolov\-sk\' a~83, 18600 Praha~8, Czech Republic}
\email{kristyna.kramer@gmail.com}

\author{Jakub Kr\'asensk\'y}
\address{Czech Technical University in Prague, Faculty of Information Technology, Department of Applied Mathematics, Thákurova~9, 16000 Praha~6, Czech Republic}
\email{jakub.krasensky@fit.cvut.cz}
\thanks{K.K. was supported by Charles University project PRIMUS/25/SCI/008.}

\keywords{universal quadratic form, quadratic lattice, ternary quadratic form, quartic number field, totally real, Kitaoka's Conjecture, lifting problem, universality criterion}
\subjclass[2020]{11E12, 11E20, 11R04, 11R16, 11R80}

\begin{document}
	
	\begin{abstract}
		We prove Kitaoka's conjecture for all totally real number fields of degree $4$ -- namely, there is no positive definite classical quadratic form in three variables which is universal. To achieve this, we study the fields (often without restricting the degree) where $2$ is a square, because in this arguably most difficult case, the recent results connecting Kitaoka's conjecture to sums of integral squares do not apply. We also prove some other properties of ternary quadratic forms over fields containing $\sqrt2$, for example in relation to the lifting problems for universal quadratic forms and for criterion sets.
	\end{abstract}
	
	\maketitle
	
	\section{Introduction}

	Lagrange's four-square theorem states that all positive integers can be written as a sum of four squares. This is considerably easier to prove than Legendre's three-square theorem (a positive integer is a sum of three squares if and only if it is not of the form $4^a(8k+7)$). In general, ternary quadratic forms tend to be the most difficult to understand -- one example of a very hard open problem is Kaplansky's conjecture that $\iks^2+2\ips^2+5\iss^2+\iks \iss$ represents all odd positive integers (proven by Rouse \cite{Rouse} under GRH).
	
	Among the big modern achievements in the theory of positive definite integral quadratic forms are the $15$-theorem of Conway--Schneeberger and the $290$-theorem of Bhargava--Hanke, which fully characterise the \emph{universal} quadratic forms, i.e., those that represent all positive integers (like the sum of four squares). A much older observation, following from the properties of Hilbert symbol, is that a positive definite quadratic form in three variables over $\Z$ cannot be universal.
	
	The situation changes if we switch to quadratic forms over number fields. All totally positive algebraic integers in $\Q(\!\sqrt5)$ are a sum of three integral squares in this field \cite{Ma}. This is the earliest and simplest example of a universal ternary quadratic form over a totally real number field. Such quadratic forms are rare, as predicted by the following influential conjecture.
	
	\begin{KitaokaConj}
		There are only finitely many totally real number fields that admit a universal ternary classical quadratic form.
	\end{KitaokaConj}
	
	Before we proceed with the historical overview, let us briefly explain the notions used in the conjecture. (For proper and complete definitions, see Section~\ref{se:prelims} or the standard book \cite{OMbook}.) Let $K$ be a totally real number field, $\OK$ its ring of integers and $\OKPlus\subset\OK$ the set of totally positive algebraic integers in $K$ (i.e., the algebraic integers that are positive in all embeddings $K\hookrightarrow \R$). An \emph{$n$-ary quadratic form} $Q$ is a homogeneous polynomial over $K$ of degree $2$ in $n$ variables. It is \emph{universal} if $Q\bigl(\OK^{\;\! n} \setminus \{\mathbf{0}\}\bigr)=\OKPlus$. Further, it is \emph{classical} if the corresponding symmetric matrix has entries in $\OK$; in particular, the coefficients of all mixed terms lie in $2\OK$. 
	
	There are many interesting results about universal quadratic forms and related topics; we refer the reader to the surveys \cite{Kim-survey,Kala-survey}. Here we focus on the results that directly concern Kitaoka's conjecture. 
	
	Chan--Kim--Raghavan \cite{CKR} proved that a real quadratic field admits a universal ternary classical quadratic form if and only if it is one of $\Q(\!\sqrt2)$, $\Q(\!\sqrt3)$, and $\Q(\!\sqrt5)$. Moreover, they provided a full list of such forms over these three fields (for $\Q(\!\sqrt2)$, we give this list in \eqref{eq:ListQFs2}). Recently, this result has been extended to not-necessarily-classical universal quadratic forms by Kala--Krásenský--Park--Yatsyna--\.{Z}mija \cite{KK+}; there are at most $11$ real quadratic fields admitting such a ternary quadratic form. Moreover, for each of these fields, an explicit probably universal quadratic form was given, but proving their universality is possibly very hard. Kim--Kim--Park \cite{KKP} proved that there are in fact only finitely many $\Q(\!\sqrt{n})$ admitting a universal $7$-ary quadratic form (later, \cite{KK+} provided an improved upper bound on $n$). This is complemented by Kim's construction of infinitely many real quadratic fields with a universal $8$-ary classical quadratic form \cite{Kim-octonary}. 
	
	It is important to bear in mind that by Earnest--Khosravani \cite{EK}, a universal ternary quadratic form can only exist in number fields of even degree (a generalisation of the fact that three variables do not suffice over $\Q$). Thus, after the solution for quadratic fields, number fields of degree $4$ are of particular interest. By Krásenský--Tinková--Zemková\footnote{The former name of the first author.} \cite{KTZ}, no real biquadratic field admits a universal ternary classical quadratic form. On the other hand, Krásenský--Scharlau announced the existence of a non-classical universal ternary quadratic form over the biquadratic field $\Q(\!\sqrt2,\sqrt5)$ as well as over another quartic field, $\Q(\zeta_{20}^{}+\zeta_{20}^{-1})$. 
	
	Beside these results in low-degree fields, there are two results valid for fields of all degrees. Kala--Yatsyna \cite{KY-weak} proved the \enquote{Weak Kitaoka's conjecture}: \emph{For any fixed $d$, there are only finitely many number fields of degree $d$ admitting a universal ternary quadratic form.} Kala--Kramer--Krásenský \cite{KKK} solves Kitaoka's conjecture for fields of odd discriminant (namely, $\Q(\!\sqrt{5})$ is the only field with a universal ternary classical quadratic form where $2$ is unramified) and also establishes a surprising connection with sums of squares: If $K \not\ni \sqrt2$ admits a universal ternary classical quadratic form, then $2\OKPlus\subset\sum^4\square$.
	
	\medskip
	
	Many of these results are in fact formulated and proven in a more general setting, namely for \emph{quadratic lattices} instead of \emph{quadratic forms}. If $K$ has class number $1$, then these two notions are equivalent, while in general, a quadratic lattice needs not correspond to a quadratic form. For example, \cite{KK+} actually proves that there are at most $13$ real quadratic fields admitting a universal ternary lattice; but the only candidates for universality over $\Q(\!\sqrt{10})$ and $\Q(\!\sqrt{65})$ are non-free and hence do not yield a universal quadratic form. From the other results, \cite{KKP,EK,KY-weak} and some parts of \cite{KKK} hold for all lattices, while the techniques used in \cite{CKR,KTZ} and other parts of \cite{KKK} require the more restricted setting of quadratic forms.
	
	\medskip
	
	In this article, we resolve Kitaoka's conjecture for all quartic fields. 
	
	\begin{theorem}\label{th:QuarticKitaoka2}
		Let $K$ be a totally real number field with $[K:\Q]=4$. Then there is no universal ternary classical quadratic form over $K$.
	\end{theorem}
	
	This will be proved as Theorem~\ref{th:QuarticKitaokaFreeLattices}. There, it is formulated for classical free lattices instead of classical quadratic forms (an equivalent statement written in a more geometric language). This leads us to the following explicit version of Kitaoka's conjecture.

	\begin{conjecture} \label{co:KitaokaStrong}
		Let $K$ be a totally real number field such that there exists a universal ternary classical quadratic form, or more generally, a universal ternary classical quadratic lattice over $K$. Then $K=\Q(\!\sqrt2)$, $\Q(\!\sqrt3)$, or $\Q(\!\sqrt5)$. 
	\end{conjecture}
	
	We discuss this conjecture further in Section~\ref{se:NotesConj}.
	
	\medskip

	It is interesting that the quadratic fields from the above conjecture can also be written as $\Q(\zeta_k^{}+\zeta_k^{-1})$ for $k=8$, $12$, and $5$, respectively. Here, $\Q(\zeta_k^{}+\zeta_k^{-1})$ is the maximal totally real subfield of the $k$th cyclotomic field $\Q(\zeta_k^{})$, where $\zeta_k^{}={\mathrm{e}}^{2\pi \mathrm{i}/k}$. Although \cite{KY-EvenBetter} is not primarily about Kitaoka's conjecture, it contains the proof that fields of this form do not admit a universal ternary classical quadratic form if $k>5$ is an odd prime or a power of $2$ other than $2^3$. These fields (for an arbitrary $k$) play a role in our other main result.
	
	\begin{theorem}[cf. Theorem~\ref{th:QFNotUniversal}] \label{th:IntroQFNotUniversal}
		Let $Q$ be a universal ternary classical quadratic form over $\Q(\!\sqrt2)$. Let $K\ni\sqrt2$ be a totally real number field containing a field of the form $\Q(\zeta_k^{}+\zeta_k^{-1})$ other than $\Q$ and $\Q(\!\sqrt2)$. Then $Q$ is not universal over $K$.
	\end{theorem}
	
	Let us put this theorem into context. The quadratic form $\iks^2+\ips^2+\iss^2$ defined over $\Q$ is universal over $\Q(\!\sqrt5)$; similarly, the non-classical quadratic form $\iks^2+\ips^2+\iss^2+\ifs^2+\iks \ifs+\ips \ifs+\iss \ifs$ is universal over the cubic field $\Q(\zeta_7^{}+\zeta_7^{-1})$. Generally, the \emph{lifting problem for universal quadratic forms} asks for field extensions $K/F$ of totally real number fields such that a quadratic form defined over $F$ is universal over $K$.
	
	The lifting problem for $F=\Q$ is well-studied \cite{KY-lifting,GT,KL,KY-EvenBetter}, but much remains open. In particular, it seems that the only proper extensions of $\Q$ which admit a universal quadratic form with coefficients in $\Z$ are $\Q(\!\sqrt5)$ and $\Q(\zeta_7^{}+\zeta_7^{-1})$. A much older instance of the lifting problem is Siegel's result \cite{Si} that the only totally real number fields where the sum of any number of squares is universal are $\Q$ and $\Q(\!\sqrt5)$. This easily implies that no other fields admit a universal diagonal quadratic form defined over $\Q$.
	
	To the best of our knowledge, the lifting problem for general $F$ was started in \cite{KY-weak} and much progress was made in \cite{KKL-Lifting}. They proved that for fixed $F$ and $d$, there are only finitely many $K\supset F$ with $[K:F]=d$ admitting a universal quadratic form defined over $F$. Also, they showed that $\Q(\!\sqrt2,\sqrt5)$ does admit a universal quadratic form with coefficients in $\Q(\!\sqrt2)$. Our Theorem~\ref{th:IntroQFNotUniversal} (and Theorem~\ref{th:QFNotUniversal}, which establishes the same conclusion under several other conditions as well) studies the lifting problem over $\Q(\!\sqrt2)$ in the particular case when the quadratic form is classical ternary. Our results can also be applied for the \emph{lifting problem for universality criterion sets}, see Corollary~\ref{co:NotCriterion}.

	\section{Preliminaries} \label{se:prelims}
	
	We write $A \subset B$ for not necessarily strict inclusion of sets, $A \subsetneq B$ for strict inclusion, and $\Z^+$ for the set of positive integers.
	
	\subsection{Algebraic integers}
	Let $K$ be a totally real number field and $\OK$ its ring of algebraic integers. We use the symbol $\square$ to denote an arbitrary square of an element of $\OK$; note that two instances of this symbol generally stand for different values. In particular, we write $\sum^n\square=\{\alpha_1^2+\cdots+\alpha_n^2\mid \alpha_i\in\OK\}$ and $\sum\square=\bigcup_{n \in \Z^+}\sum^n\square$.
	
	\medskip
	
	Assume that $[K:\Q]=d$; then there exist $d$
	distinct embeddings $\sigma_i: K\hookrightarrow\R$ for $1\leq i\leq d$. The \emph{norm} and the \emph{trace} of $\alpha\in K$ are defined as
	\[
	\norm{K/\Q}{\alpha}=\prod_{i=1}^d\sigma_i(\alpha), \qquad \Tr{K/\Q}{\alpha}=\sum_{i=1}^d\sigma_i(\alpha).
	\]
	We also define the \emph{house} of $\alpha \in K$ as
	\[
	\house{\alpha} = \max_{1\leq i\leq d} \abs{\sigma_i(\alpha)};
	\]
	note that in contrast to the norm and trace, the house of $\alpha\in\OK$ is typically not a rational integer.
	
	Let $\alpha,\beta\in K$; we write $\alpha\succ\beta$ if $\sigma_i(\alpha)>\sigma_i(\beta)$ for all $1\leq i \leq d$, and $\alpha\succeq\beta$ if $\alpha\succ\beta$ or $\alpha=\beta$. We say that $\alpha$ is \emph{totally positive} if $\alpha\succ0$. We denote by $\OKPlus$ the set of all totally positive algebraic integers in $K$.
	
	A \emph{unit} is an element of $\OK$ with norm equal to $\pm1$; we denote the group of all units by $\UK$. Furthermore, we write $\UKPlus$ for the subgroup of totally positive units and $\UKctv$ for the subgroup of squares of units; note that $\UKctv\subset\UKPlus$.
	
	For $\alpha,\beta\in\OK$, we write $\beta \mid \alpha$ if $\alpha/\beta \in \OK$. We call $\alpha$ \emph{squarefree} if $\omega^2 \mid \alpha$ implies $\omega\in\UK$; we point out that the condition is about divisibility by a square of an element, \emph{not} of a general ideal. Section~\ref{se:notSquarefree} studies fields where $2$ is not squarefree.
	
	An element $\alpha\in\OKPlus$ is \emph{indecomposable} if it cannot be written as $\alpha=\beta+\gamma$ with $\beta,\gamma\in\OKPlus$; otherwise, $\alpha$ is called \emph{decomposable}.
	
	\begin{lemma}[{\cite[Lemma~2.1b]{KY-lifting} and \cite[Thm.~2.2]{KY-EvenBetter}}] \label{le:smallAreIndecomposable} 
		Let $K$ be a totally real number field with $[K:\Q]=d$. Let $\alpha \in \OKPlus$.
		\begin{enumerate}
			\item If $\norm{K/\Q}{\alpha} < 2^d$, then $\alpha$ is indecomposable. 
			\item If $\Tr{K/\Q}{\alpha} <\frac52d$, then $\alpha$ is indecomposable or $\alpha=2$.
		\end{enumerate}
	\end{lemma}
	
	Although $2$ decomposes as $1+1$, one can show that this is the only possible decomposition as a sum of two totally positive algebraic integers. In particular, if $\sqrt2\notin K$, then the only $\omega\in\OK$ satisfying $\omega^2\preceq2$ are $0$ and $\pm1$.
	
	There are always only finitely many indecomposable elements up to multiplication by squares of units. For example, in $\Q(\!\sqrt2)$, the element $\lambda=2+\sqrt2$ is indecomposable by the lemma; and in fact, the set of indecomposables is exactly $\UKctv[\Q(\!\sqrt2)] \cup \lambda\,\UKctv[\Q(\!\sqrt2)]$.
	
	\subsection{Signatures and units}
	
	The size of the factor group $\UKPlus/\UKctv$ plays a significant role in the study of Kitaoka's conjecture. It is always a power of $2$ and it can be computed from the \emph{class number} $h(K)$ and \emph{narrow class number} $h^+(K)$ as
	\[
	\abs{\UKPlus/\UKctv} = \frac{h^+(K)}{h(K)}.
	\]
	A significant advantage of this formula is that both class numbers can be found in the database LMFDB \cite{LMFDB}.
	
	A useful fact (see, e.g., \cite[Rem.~1]{DDK}) is that the size of this group cannot decrease under field extension:
	
	\begin{lemma} \label{le:unitsDontDissappear}
		Let $K \supset F$ be totally real number fields. Then $\abs{\UKPlus/\UKctv} \geq \abs{\UKPlus[F]/\UKctv[F]}$.
	\end{lemma}
	
	If $\sigma_1,\dots,\sigma_d$ are the embeddings $K\hookrightarrow\R$, then the \emph{signature} of a nonzero $\alpha\in K$ is the $d$-tuple 
	\[
	\ms(\alpha)=\bigl(\sign\sigma_1(\alpha),\dots,\sign\sigma_d(\alpha)\bigr)\in \{\pm1\}^d.
	\]  
	For example, $\alpha \succ 0$ if and only if $\ms(\alpha)=(+1, \ldots, +1)$. 
	
	The case when $\UKPlus=\UKctv$ (i.e., $\abs{\UKPlus/\UKctv}=1$), which is in the centre of our attention, can also be characterised by the fact that \enquote{there exist units of all signatures}. This has an important corollary: \emph{Let $\UKPlus=\UKctv$. Then for every nonzero $\beta\in\OK$ there is some $\eta\in\UK$ such that $\eta \beta \in \OKPlus$.} Sometimes we use $\eta_{\beta}$ to denote such a unit.

	\subsection{Quadratic forms and lattices}
	Let $V$ be a finite-dimensional vector space over $K$. The map $\varphi:V\to K$ is called a \emph{quadratic map} if the following conditions are satisfied:
	\begin{enumerate}[(1)] \setlength{\itemsep}{2pt}
		\item $\varphi(\alpha\vv)=\alpha^2\varphi(\vv)$ for all $\alpha\in K$ and $\vv\in V$;
		\item the map $B_\varphi:V\times V\to K$ defined by $B_\varphi(\vv,\vw)=\frac12\bigl(\varphi(\vv+\vw)-\varphi(\vv)-\varphi(\vw)\bigr)$ is bilinear.
	\end{enumerate}
	The pair $(V,\varphi)$ is a \emph{quadratic space} over $K$. It is \emph{positive definite} if $\varphi(\vv)\succ0$ for all nonzero $\vv\in V$. (Often, it is called \enquote{totally positive definite} instead.)

	\medskip
	
	Let $(V,\varphi)$ be a quadratic space and let $(\vv_1,\dots,\vv_n)$ be a basis of $V$; then 
	\begin{equation}\label{eq:QMapToForm}
		Q(\iks_1,\dots,\iks_n)=\varphi(\iks_1\vv_1+\dots+\iks_n\vv_n)
	\end{equation}
	is an \emph{$n$-ary quadratic form}. In particular, it is a homogeneous polynomial of degree $2$ in $n$ variables with coefficients in $K$, i.e., it can be written as
	\[
	Q(\iks_1,\dots,\iks_n)=\sum_{1\leq i\leq j\leq n}\alpha_{ij}\iks_i\iks_j \quad \text{where } \alpha_{ij}\in K.
	\]
	A quadratic form $Q$ is \emph{integral} if $\alpha_{ij}\in\OK$ for all $1\leq i\leq j\leq n$, and \emph{classical} if it is integral and $\alpha_{ij}\in2\OK$ for all $i\neq j$. It is \emph{positive definite} if the corresponding quadratic space is. It \emph{represents} $\alpha\in K$ if $\alpha \in Q(\OK^{\;\! n})$. A positive definite integral quadratic form is \emph{universal} (\emph{over $K$}) if it represents all elements of $\OKPlus$. 
	
	We say that $n$-ary quadratic forms $Q_1, Q_2$ are \emph{isometric} if there exists $M\in \OK^{\;\! n\times n}$ with $\det M \in\UK$, such that $Q_2(\boldsymbol{\iks})=Q_1(M\boldsymbol{\iks})$, where $\boldsymbol{\iks}=(\iks_1,\dots,\iks_n)^{\mathrm{T}}$. We denote this by $Q_1\simeq Q_2$.
	If $Q_1$ is an $m$-ary and $Q_2$ an $n$-ary integral quadratic form, where $m\geq n$, then $Q_1$ \emph{represents} $Q_2$ if there exists $M \in \OK^{\;\! m\times n}$ such that $Q_2(\boldsymbol{\iks}) = Q_1(M\boldsymbol{\iks})$. For example, $3x^2+2xy+y^2=x^2+x^2+(x+y)^2$ is represented by $X^2+Y^2+Z^2$.
	
	\medskip
	
	Let $V$ be a finite-dimensional vector space over $K$. Then $L$ is a \emph{lattice on $V$} if it is a finitely generated $\OK$-submodule of~$V$ and $KL=V$. The \emph{rank} of $L$ is then $\rank L = \dim V$. Lattices of rank $1,2,3$ are called \emph{unary}, \emph{binary} and \emph{ternary}, respectively.
	
	For every lattice $L$ of rank $n$, we can find linearly independent vectors $\vv_1, \ldots, \vv_n\in L$ and an ideal $\ia \subset \OK$ such that
	\[
	L = \OK \vv_1 \oplus \cdots \oplus \OK \vv_{n-1} \oplus \ia^{-1} \vv_n
	\]
	(see the structure theorem for finitely generated modules over Dedekind domains); we refer to $(\vv_1,\dots,\vv_n)$ as a \emph{pseudobasis} of $L$. (Alternatively, we can rewrite $\ia^{-1}\vv_n$ as $\mathfrak{b}\vv_n'$ where $\mathfrak{b}$ is an ideal and $\vv_n'\in KL$; the term \emph{pseudobasis} covers that situation as well, and we do not need its precise definition.) We say that $L$ is \emph{free} if it is free as an $\OK$-module, which occurs if and only if $\ia$ is a principal ideal; in such case, we can scale $\vv_n$ so that $\ia=\OK$, and then $(\vv_1,\dots,\vv_n)$ is a \emph{basis} of $L$. In particular, all lattices are free when $h(K)=1$, which applies, e.g., when $K=\Q(\!\sqrt2)$.
	
	\medskip
	
	A \emph{quadratic lattice} over $K$ is a tuple $(L,\varphi)$, where $L$ is a lattice on $V$ and $(V,\varphi)$ is a quadratic space over $K$. It is called \emph{positive definite} if the quadratic space $(V,\varphi)$ is positive definite. We say that the quadratic lattice $(L,\varphi)$ is \emph{integral} if $\varphi(L)\subset\OK$. It is \emph{classical} if $B_\varphi(\vu, \vv) \in \OK$ for all $\vu, \vv \in L$; in particular, a classical quadratic lattice is always integral.

	\begin{remarkunnum}
		Alternatively, we could define a quadratic lattice $(L,\varphi)$ without using a quadratic space: $L$ is a torsion-free finitely generated $\OK$-module and $\varphi:L\to K$ is a quadratic map. This is the approach we took in \cite{KKK}. The definitions are equivalent; one recovers the quadratic space by taking the tensor product as $V=K \otimes L$, and $\varphi$ extends uniquely to $V$. We often use this fact implicitly when we speak about lattices without referring to the quadratic space. 
	\end{remarkunnum} 
	
	\begin{convention}
		We fix some not entirely standard terminology for this paper.
		\begin{enumerate}[wide=0pt, listparindent=\parindent]\setlength{\itemsep}{1pt}
			\item We will slightly abuse the notation and not distinguish between quadratic forms and quadratic maps. 
			\item All our lattices will be quadratic, and we will usually denote them simply by $L$; unless specified otherwise, the corresponding quadratic map will always be denoted by $Q$. We usually use $KL$ instead of $V$ for the underlying quadratic space.
			\item All our lattices will be positive definite. 
			\item Most but not all of our lattices will be classical. Therefore, we will always say explicitly if the lattice is classical. Note that classical lattices are necessarily integral. On the other hand, by a \emph{lattice} we mean any quadratic lattice that can be non-integral. 
		\end{enumerate}
	\end{convention}

	A lattice $L$ \emph{represents} an element $\alpha \in K$ if there exists a $\vv \in L$ such that $Q(\vv) = \alpha$. A classical (or, more generally, integral) lattice $L$ over $K$ is \emph{universal} if it represents all of $\OKPlus$.

	Two lattices $L_1$ and $L_2$ are called \emph{isometric}, written $L_1 \simeq L_2$, if there exists a $K$-linear bijection $\iota : L_1 \to L_2$ preserving the quadratic map.
	
	\medskip
	
	Free quadratic lattices are in one-to-one correspondence with quadratic forms $\OK^{\;\! n} \to K$ via the same relation as in \eqref{eq:QMapToForm}, where $(\vv_1, \ldots, \vv_n)$ is a basis of $L$ instead of $V$. (Strictly speaking, to get a one-to-one correspondence, one should consider classes modulo isometry.) Note that a free quadratic lattice is integral / classical / positive definite / universal if and only if the corresponding quadratic form is.

	\medskip
	
	The \emph{orthogonal sum} of two lattices $(L_1, Q_1)$ and $(L_2, Q_2)$ is denoted $L_1 \perp L_2$ and defined as the direct sum $L_1 \oplus L_2$ of the underlying modules equipped with the quadratic map $Q(\vv_1 \oplus \vv_2) = Q_1(\vv_1) + Q_2(\vv_2)$. We write $\qf{\alpha}$ for the unary free lattice $(\OK \vv, Q)$ where $Q(\vv)=\alpha$, and $\qf{\alpha_1,\ldots,\alpha_n}$ for $\qf{\alpha_1}\perp\ldots\perp\qf{\alpha_n}$. A lattice that can be written in this form is sometimes called \emph{diagonalisable}.
	
	\begin{lemma}[{\cite[Prop.~3.4]{Kala-survey}}] \label{le:UnitSplits}
		Let $K$ be a totally real number field and let $L$ be a classical lattice over $K$. If $L$ represents some $\ve\in\UKPlus$, then there exists a classical lattice $L'$ over $K$ such that $L\simeq\qf{\ve}\perp L'$.
	\end{lemma}
	
	\subsection{Gram matrix}
	
	Let $(V, Q)$ be a quadratic space over $K$ and $\vv_1,\dots,\vv_k\in V$. The \emph{Gram matrix} of $\vv_1,\dots,\vv_k$ is the symmetric matrix $G=\bigl(B_Q(\vv_i,\vv_j)\bigr)_{1\leq i,j\leq k}$. Note that since $Q$ is positive definite, as we assume throughout the text, the matrix $G$ is totally positive semidefinite (i.e., positive semidefinite in all embeddings). Thus we have
	\[
	B_Q(\vv_i,\vv_j)^2\preceq Q(\vv_i)\cdot Q(\vv_j) \quad \text{for } 1\leq i, j \leq k;
	\]
	this is also known as \emph{Cauchy--Schwarz inequality}. Moreover, $G$ is totally positive definite (i.e., $\det G \neq 0$) if and only if the vectors $\vv_1,\dots,\vv_k$ are linearly independent over $K$. More generally, it holds that $\rank (\OK\vv_1 + \cdots + \OK\vv_k) = \rank G$.
	
	If $L$ is a free quadratic lattice on $V$ and $(\vv_1,\dots,\vv_k)$ is a basis of $L$, then we call $G$ the \emph{Gram matrix of $L$ with respect to $(\vv_1,\dots,\vv_k)$}. We put $\det L = \det G$; it is well-defined up to multiplication by $\UKctv$. Note that the Gram matrix of a lattice (with respect to a basis) determines the lattice up to isometry.
	
	Let $G \in \OK^{\;\! n \times n}$ be totally positive definite. Then we define the classical free lattice $\qf{G}$ as $(\OK^{\;\! n},Q)$ where $Q(\vv)=\vv^{\mathrm{T}}G\vv$. Then $G$ is the Gram matrix of $\qf{G}$ with respect to the standard basis.

	\subsection{Sublattices and dual lattices}\label{ss:Dual}
	
	Let $L$ and $L'$ be quadratic lattices on the same quadratic space $(V,Q)$. In this text, we say that $L'$ is a \emph{sublattice} of $L$ if $L'\subset L$ as sets; then necessarily $KL=KL'$. In particular, $\rank L=\rank L'$. For example, $\qf{1,1,4}$ is isometric to a sublattice of $\qf{1,1,1}$, but $\qf{1,1}$ is not isometric to a sublattice of $\qf{1,1,4}$.
	
	If $L'$ is a sublattice of $L$, then we say that $L$ is an \emph{overlattice} of $L'$.
	
	\medskip
	
	Let $(L,Q)$ be a lattice over $K$, and $V=KL$ the vector space generated by $L$. We define the \emph{dual} lattice of $L$ as
	\[
	\dual{L}=\{\vv\in V \mid B_Q(\vv,\vx)\in\OK \text{ for all } \vx\in L\}.
	\]
	Note that $\dual{L}$ is a lattice, but it does not have to be classical (even if $L$ is classical). Typically, $\dual{L}$ is not even integral. Although $L$ and $\dual{L}$ share the same quadratic space, they are usually not isometric.
	
	\begin{lemma} \label{le:DualLattice}
		Let $L$ be a classical lattice over $K$. Then $L\subset \dual{L}$ and all classical overlattices of $L$ are sublattices of $\dual{L}$.
	\end{lemma}
	
	\begin{proof}
		Let $\vv\in L$. Then $B_Q(\vv,L) \subset \OK$, since $L$ is classical; thus $\vv \in \dual{L}$. This proves the first assertion. Now let $\tilde{L}$ be a classical overlattice of $L$ and take $\vw \in \tilde{L}$. Then $B_Q(\vw,L) \subset B_Q(\vw,\tilde{L}) \subset \OK$; thus $\vw \in \dual{L}$, proving the second assertion.
	\end{proof}
	
	In fact, a lattice $L$ is classical \emph{if and only if} $L \subset \dual{L}$. Also, for $\vw\in KL$, one can easily see that $\vw \in \dual{L}$ together with $Q(\vw) \in \OK$ is equivalent to $L + \OK\vw$ being classical; but not all sublattices of $\dual{L}$ are classical, since $\dual{L}$ itself often is not. Furthermore, a lattice, even a classical lattice, can have infinitely many overlattices; however, it follows from Lemma~\ref{le:DualLattice} that a classical lattice has only finitely many classical overlattices.
	
	A lattice $L$ over $K$ is \emph{unimodular} if $L=\dual{L}$. Note that by Lemma~\ref{le:DualLattice}, unimodular lattices do not have any proper classical overlattices. A free lattice is unimodular if and only if it is classical and $\det L\in\UK$.
	
	\medskip
	
	The following well-known lemma explicitly describes the dual of a classical free lattice both as a set (giving the basis in \ref{le:DualGramMatrix-basis} and \ref{le:DualGramMatrix-coeff}), and as a quadratic form (part \ref{le:DualGramMatrix-Gram}).
	
	\begin{lemma} \label{le:DualGramMatrix}
		Let $L$ be a classical free lattice over $K$ with basis $(\e_1,\dots,\e_n)$, i.e., $L=\OK\e_1+\dots+\OK\e_n$. Let $G=(g_{ij})_{1\leq i,j\leq n}$ be the Gram matrix of $L$ with respect to this basis, i.e., $g_{ij}=B_Q(\e_i,\e_j)$. Then:
		\begin{enumerate}
			\item The dual lattice $\dual{L}$ is free with basis $(\f_1,\dots,\f_n)$ satisfying $B_Q(\f_i,\e_j)=\delta_{ij}$ for all $1\leq i,j\leq n$.\label{le:DualGramMatrix-basis}
			\item Let $h_{ij}\in K$ be such that $\f_j=\sum_{i=1}^nh_{ij}\e_i$. Then $(h_{ij})_{1\leq i,j\leq n}=G^{-1}$. \label{le:DualGramMatrix-coeff}
			\item The Gram matrix of $\dual{L}$ with respect to $(\f_1,\dots,\f_n)$ is $G^{-1}$. \label{le:DualGramMatrix-Gram}
		\end{enumerate}
	\end{lemma}
	\begin{proof}
		Part~\ref{le:DualGramMatrix-basis} follows from \cite[\S82.F]{OMbook}.
		
		To prove \ref{le:DualGramMatrix-coeff}, denote $H=(h_{ij})_{1\leq i,j\leq n}$ and compute the element at the $i${th} row and $j${th} column in the product $GH$:
		\[
		\sum_{k=1}^ng_{ik}h_{kj}=\sum_{k=1}^nB_Q(\e_i,\e_k)h_{kj}=B_Q\Bigl(\e_i,\sum_{k=1}^nh_{kj}\e_k\Bigr)=B_Q(\e_i,\f_j)=\delta_{ij}.
		\]
		Hence, $GH=I_n$; it follows that $H=G^{-1}$.
		
		For \ref{le:DualGramMatrix-Gram}, denote $P=(p_{ij})_{1\leq i,j\leq n}$ the Gram matrix of $\dual{L}$ with respect to $(\f_1,\dots,\f_n)$. Then:
		\[
		p_{ij}
		=B_Q(\f_i,\f_j)
		=B_Q\Bigl(\f_i,\sum_{k=1}^nh_{kj}\e_k\Bigr)
		=\sum_{k=1}^nh_{kj}B_Q(\f_i,\e_k)
		=\sum_{k=1}^nh_{kj}\delta_{ik}
		=h_{ij}. \qedhere
		\]
	\end{proof}
	
	\subsection{Some useful results}
	
	The following claim is essentially (but not explicitly) contained in \cite{KKK}.
	
	\begin{proposition} \label{pr:KKK}
		If $K\ni\sqrt2$ is a totally real number field (of any degree) admitting a universal ternary classical lattice, then $\UKPlus=\UKctv$.
	\end{proposition}
	\begin{proof}
		We need to show that $\abs{\UKPlus/\UKctv}=1$. If $\abs{\UKPlus/\UKctv}>2$, then we can apply the trivial \cite[Prop.~2.5]{KKK} to get the non-existence of a universal ternary classical lattice. If $\abs{\UKPlus/\UKctv}=2$, then we apply \cite[Thm.~3.1(f)]{KKK}.
	\end{proof}
	
	The following lemma is useful for bounding the discriminant of a quartic field by the house of its generator. The corresponding result for any number of variables is contained in \cite{Schur} on page 378 as statement I.; we provide a proof for the reader's convenience.
	
	\begin{lemma}\label{le:4variables}
		The maximum of the function
		\[
		g(\iks_1,\iks_2,\iks_3,\iks_4) = \abs{\iks_4-\iks_3}\abs{\iks_4-\iks_2}\abs{\iks_4-\iks_1}\abs{\iks_3-\iks_2}\abs{\iks_3-\iks_1}\abs{\iks_2-\iks_1}
		\]
		on the set $[-1,1]^4 \subset \R^4$ is $2^6\cdot5^{-5/2}$.
	\end{lemma}
	\begin{proof}
		Clearly, to maximise the value, one of the variables must be $-1$ and one must be $+1$. Indeed, if the largest variable is replaced by $1$ while the others are kept constant, the value of $g$ increases; similarly for the smallest variable and $-1$.
		
		So it is enough to find the maximum and minimum of the function
		\begin{align*}
			f(\iks,\ips) &= 2(1-\iks)(1-\ips)(-1-\iks)(-1-\ips)(\iks-\ips) \\
			&= 2(1-\iks^2)(1-\ips^2)(\iks-\ips)
		\end{align*}
		for $|\iks|,|\ips| \leq 1$. By computing the partial derivatives, we get the following system of equations:
		\[
		2(1-\ips^2)(1+2\iks \ips-3\iks^2) = 0 \qquad \text{and} \qquad 2(1-\iks^2)(-1-2\iks\ips+3\ips^2)=0.
		\]
		We are not interested in the trivial local extremes where $|\iks|=1$ or $|\ips|=1$, which correspond to minima and not maxima of $g$, so we can ignore the first brackets; it remains to solve the system
		\[
		1+2\iks\ips-3\iks^2 = 0 \qquad \text{and} \qquad -1-2\iks\ips+3\ips^2=0.
		\]
		By summing the two equations and again ignoring the trivial local extreme where $\iks=\ips$, we obtain the condition $\iks+\ips=0$. Plugging it back into either equation yields $1-2\iks^2-3\iks^2=0$, so $\iks=-\ips=\pm\frac{1}{\sqrt5}$. All in all, the original function $g$ has a maximum at those points where the four variables all take distinct values from $\bigl\{\pm1, \pm\frac{1}{\sqrt5}\bigr\}$. (Recall that a global maximum on a compact set must exist, and this is the only candidate.) 
	\end{proof}
	
	\begin{proposition} \label{pr:KubasBound}
		Let $K=\Q(\beta)$ be a quartic totally real number field, $\beta\in\OK$. Then
		\[
		\disc K\leq\frac{2^{12}}{5^5}\house{\beta}^{12}.
		\]
	\end{proposition}
	\begin{proof}
		To estimate the discriminant of the field, we use the value $\Delta_{K/\Q}(\beta)$ called the \emph{discriminant} of $\beta$, which is the square of the determinant of the matrix $\bigl(\sigma_i(\beta^{j-1})\bigr)_{1\leq i,j\leq4}$. Then ${\disc K\leq\Delta_{K/\Q}(\beta)}$.
		
		Let $\beta_i = \sigma_i(\beta)$ be the conjugates of $\beta$. Using the Vandermonde determinant, we can express the discriminant of $\beta$ as
		\[
		\Delta_{K/\Q}(\beta) = \bigl((\beta_4-\beta_3)(\beta_4-\beta_2)(\beta_4-\beta_1)(\beta_3-\beta_2)(\beta_3-\beta_1)(\beta_2-\beta_1) \bigr)^2.
		\]
		Put $x_i := \beta_i / \;\! \house{\beta}$. Then $-1\leq x_i \leq 1$ for all $i$, and we get
		\[
		\Delta_{K/\Q}(\beta) = \bigl((x_4-x_3)(x_4-x_2)(x_4-x_1)(x_3-x_2)(x_3-x_1)(x_2-x_1) \bigr)^2\,\house{\beta}^{12}.
		\]
		By Lemma~\ref{le:4variables}, we know the maximal possible absolute value of the bracket, which yields
		\[
		\Delta_{K/\Q}(\beta) \leq \frac{2^{12}}{5^5} \house{\beta}^{12}. \qedhere
		\]
	\end{proof}
	
	\section{Lifting from \texorpdfstring{$\Q(\!\sqrt2)$}{Q(sqrt2)}} \label{se:lifting}
	
	In this section, we look at the ternary classical quadratic forms that are universal over $\Q(\!\sqrt2)$. By \cite{CKR}, these are precisely the quadratic forms that are isometric to one of the following four:
	\begin{equation}\begin{array}{l} \label{eq:ListQFs2}
			Q_1(\iks,\ips,\iss)=\iks^2+\ips^2+(2+\sqrt2)\iss^2,\\ [2mm]
			Q_2(\iks,\ips,\iss)=\iks^2+(2+\sqrt2)\ips^2+2\ips\iss+(2-\sqrt2)\iss^2,\\[2mm]
			Q_3(\iks,\ips,\iss)=\iks^2+(2+\sqrt2)\ips^2+2\ips\iss+3\iss^2,\\[2mm]
			Q_3'(\iks,\ips,\iss)=\iks^2+(2-\sqrt2)\ips^2+2\ips\iss+3\iss^2.
	\end{array}\end{equation}
	
	We lift these forms to a larger field and ask whether they are still universal. By \cite[Thm.~1.1]{KKL-Lifting}, for every fixed degree $d$, there are at most finitely many fields $K$ with $[K:\Q(\!\sqrt2)]=d$ over which any of these quadratic forms is universal. We did not find any such field -- on the contrary, we give several conditions on $K$ under which the forms are not universal; see Theorem~\ref{th:QFNotUniversal}. In fact, we expect that none of these forms is universal over any proper extension of $\Q(\!\sqrt2)$, because quite possibly all fields of degree $> 4$ satisfy condition \ref{th:QFNotUniversal-2OK} from Theorem~\ref{th:QFNotUniversal}. Note that this conjecture about lifting would also immediately follow from our strengthened form of Kitaoka's conjecture~\ref{co:KitaokaStrong}.
	
	Moreover, if a quadratic form universal over a totally real number field $F$ fails to be universal over an extension of $F$, it has implications on the so-called universality criterion set.
	
	\begin{definition}
		Let $K$ be a totally real number field. We say that $S\subset\OKPlus$ is a \emph{universality criterion set} if for every classical quadratic form $Q$ over $K$ we have the following equivalence: $Q$ is universal if and only if $Q$ represents all $\alpha\in S$.
	\end{definition}
	
	Let $F\subsetneq K$ be totally real number fields. Clearly, $\OKPlus[F]$ is a universality criterion set for $F$. But can $\OKPlus[F]$ be a universality criterion set for $K$? This question, which can naturally be called \emph{lifting problem for universality criterion sets}, was asked in \cite[Section~6]{KKR-CriterionSets}. The authors have shown that if $F=\Q$ or $F=\Q(\!\sqrt5)$, then it is not possible for any $K$. Their seemingly easy proof of this claim is based on the very strong result of Siegel \cite{Si} that the sum of any number of squares is not universal over any field apart from $\Q$ and $\Q(\!\sqrt5)$. 
	
	We study the analogous question for extensions of $F=\Q(\!\sqrt2)$. As an application of Theorem~\ref{th:QFNotUniversal}, we show for many fields $K$ that $\OKPlus[F]$ is not a universality criterion set for $K$ (see Corollary~\ref{co:NotCriterion}).
	
	\subsection{Sums of squares} 
	
	We start by showing that if $Q$ is one of the quadratic forms in \eqref{eq:ListQFs2}, then anything that is represented by $2Q$ is also represented by the sum of four squares (three in the case of $Q_2$). This will be essential for the proof of part \ref{th:QFNotUniversal-2OK} of Theorem~\ref{th:QFNotUniversal}.
	
	\begin{lemma} \label{le:2Q}
		Let $Q$ be one of the quadratic forms in \eqref{eq:ListQFs2}. Then $2Q$ is represented by the quadratic form $\iks_1^2+\iks_2^2+\iks_3^2+\iks_4^2$.
	\end{lemma}
	\begin{proof}
		We have
		\begin{align*}
			2Q_1(\iks,\ips,\iss) &= (\!\sqrt2\iks)^2 + (\!\sqrt2\ips)^2 + \bigl((1+\sqrt2)\iss\bigr)^2 + \iss^2,\\
			2Q_2(\iks,\ips,\iss) &= (\!\sqrt2\iks)^2+\bigl((1+\sqrt2)\ips+(-1+\sqrt2)\iss\bigr)^2+(\ips+\iss)^2,\\
			2Q_3(\iks,\ips,\iss)&= (\!\sqrt2\iks)^2 + \bigl((1+\sqrt2)\ips\bigr)^2+(\ips+2\iss)^2+(\!\sqrt2\iss)^2,\\ 
			2Q_3'(\iks,\ips,\iss)&= (\!\sqrt2\iks)^2 + \bigl((1-\sqrt2)\ips\bigr)^2+(\ips+2\iss)^2+(\!\sqrt2\iss)^2. \qedhere
		\end{align*}  
	\end{proof}
	
	\begin{remark} \label{re:2Q2}
		From the proof we see that $2Q_2$ is in fact represented by $\iks_1^2+\iks_2^2+\iks_3^2$.
	\end{remark}

	\subsection{Indecomposable elements}
	
	The aim is to give a description of all indecomposable elements in fields over which at least one of the quadratic forms from \eqref{eq:ListQFs2} is universal. 
	
	In the following lemma, we need a slight generalisation of the term \enquote{indecomposable}. Recall that $\ms(\alpha)$ is the signature of $\alpha$. We say that $\alpha\in\OK$ is \emph{$\ms$-indecomposable}, if there do not exist $\beta, \gamma\in\OK$ with $\ms(\beta)=\ms(\gamma)$ such that $\alpha=\beta+\gamma$. Note that if $\alpha\in\OKPlus$, then $\alpha$ is $\ms$-indecomposable if and only if it is indecomposable.

	\begin{lemma}\label{le:SigmaIndecOfSquares}
		Let $\alpha\in\OKPlus$ be indecomposable and $\alpha=\beta\tau^2$ for some $\beta, \tau\in\OK$. Then $\beta$ is indecomposable and $\tau$ is $\ms$-indecomposable.
	\end{lemma}
	\begin{proof}
		Clearly, $\beta$ is totally positive. First, suppose that it is decomposable: $\beta=\beta_1+\beta_2$ for some $\beta_1,\beta_2\in\OKPlus$. Then $\alpha=\beta_1\tau^2+\beta_2\tau^2$, i.e., $\alpha$ is decomposable.
		
		Now assume that $\tau=\tau_1+\tau_2$ for some $\tau_1,\tau_2\in\OK$ such that $\ms(\tau_1)=\ms(\tau_2)$. Note that then $\tau_1\tau_2\in\OKPlus$. We get $\alpha=\beta\tau_1^2+2\beta\tau_1\tau_2+\beta\tau_2^2$, i.e., $\alpha$ is decomposable.
	\end{proof}
	
	The $\ms$-indecomposable elements are particularly easy to understand if $\UKPlus=\UKctv$: In that case, there exist units of all signatures, so for every nonzero $\beta\in\OK$ there is some $\eta\in\UK$ such that $\eta \beta \in \OKPlus$; then $\beta$ is $\ms$-indecomposable if and only if $\eta\beta$ is indecomposable. In the following lemma, we show that we can \enquote{unsquare} any indecomposable element.
	
	\begin{lemma} \label{le:UnsquaringIndecomposables}
		Let $K$ be a totally real number field with $\UKPlus=\UKctv$. Then for every nonzero $\alpha\in\OK \setminus\UK$ we can find $\mu\in\UK$ and nonsquare $\beta\in\OKPlus$ such that $\alpha=\mu\beta^{2^k}$ for some $k\geq0$.  Moreover, if $\alpha$ is totally positive and indecomposable, then $\beta$ is indecomposable as well.
	\end{lemma}
	\begin{proof}
		Let $\alpha_0=\eta_0\alpha$ with $\eta_{0}\in\UK$ such that $\alpha_0\succ0$. We iteratively define sequences of $\alpha_i$'s and $\eta_i$'s: for each $i\geq0$, if $\alpha_i$ is not a square, then we set $\beta=\alpha_i$ and $k=i$, and we stop. Otherwise, thanks to the existence of units of all signatures, we can put $\alpha_{i+1}=\eta_{i+1}\sqrt{\alpha_i}$ with $\eta_{i+1}\in\UK$ chosen so that $\alpha_{i+1}\succ0$. Then $\norm{K/\Q}{\alpha_{i+1}}=\norm{K/\Q}{\alpha_i}^{1/2}$. In particular, since the norm is an integer greater than $1$, it decreases in each step, so the process is finite. It follows that 
		\[
		\beta^{2^k}=\alpha_k^{2^k}=\alpha\prod_{i=0}^k\eta_i^{2^i}.
		\]
		Setting $\mu^{-1}=\prod_{i=0}^k\eta_i^{2^i}$, the claim follows.
		
		Finally, using the equality $\alpha=\mu\beta^{2^k}$, we see that if $\beta$ is decomposable, then so is $\alpha$.
	\end{proof}
	
	The following lemma, although simple, provides an interesting insight into the possible structure of indecomposables in a totally real number field $K$. We remark that one could easily formulate a generalisation where $2+\sqrt2$ is replaced by another element of small norm or trace.
	
	\begin{lemma} \label{le:indecStrengthen}
		Let $K\ni\sqrt2$ be a totally real number field where every indecomposable element is of the form $\square$ or $(2+\sqrt2)\square$. Then the set of indecomposable elements in $K$ is precisely $\UKctv \cup (2+\sqrt2)\UKctv$. 
	\end{lemma}
	\begin{proof}
		Note that all elements of $\UKctv\cup(2+\sqrt2)\UKctv$ are indecomposable by Lemma~\ref{le:smallAreIndecomposable}; thus, it suffices to prove that all indecomposables lie in this set. Also, all totally positive units are indecomposable by the same lemma, so the assumption immediately yields that they are squares, i.e., $\UKPlus=\UKctv$.
		
		Suppose that there exists an indecomposable element $\alpha^2$ or $(2+\sqrt2)\alpha^2$ with $\alpha\in\OK\setminus\UK$. Note that $\alpha^2$ must be indecomposable in both cases. Applying Lemma~\ref{le:UnsquaringIndecomposables} on $\alpha^2$, we can find $\mu\in\UK$ and a nonsquare indecomposable $\beta\in\OKPlus$ such that $\alpha^2=\mu\beta^{2^k}$ for some $k\geq0$. Since $\beta$ is indecomposable and nonsquare, we must have $\beta=(2+\sqrt2)\gamma^2$ for some $\gamma\in\OK$. Then
		\[
		\alpha^2=\mu\beta^{2^k}=\mu(2+\sqrt2)^{2^k}\gamma^{2^{k+1}}.
		\]
		However, $(2+\sqrt2)^{2^k}$ is not indecomposable for $k\geq1$ and the choice $k=0$ gives $\beta=\square$ (since $\mu\in\UKPlus=\UKctv$), which is a contradiction. Thus $\alpha\in\UK$ for every indecomposable $\alpha^2$ or $(2+\sqrt2)\alpha^2$, which proves the claim.
	\end{proof}
	
	We are ready to characterise indecomposable elements in the fields where at least one of the quadratic forms in \eqref{eq:ListQFs2} is universal. 
	
	\begin{proposition} \label{pr:IndecAboveSqrt2}
		Let $K\ni\sqrt2$ be a totally real number field. Assume that at least one of the quadratic forms in \eqref{eq:ListQFs2} is universal over $K$. Then the set of indecomposable elements in $K$ is precisely $\UKctv \cup (2+\sqrt2)\UKctv$. 
	\end{proposition}
	\begin{proof}
		In this proof only, we use $\OKctv$ to denote the set of elements $\alpha^2$ with $\alpha\in\OK$. Furthermore, we write $\OK^{+,0}$ for $\OKPlus \cup \{0\}$. 
		
		\begin{enumerate}[(1),wide=\parindent, listparindent=\parindent]\setlength{\itemsep}{6pt}    
			\item \label{pr:IndecAboveSqrt2-1}
			First, assume that $Q_1$ is universal over $K$. Then it represents all indecomposable elements in $K$, so they are of the form $\square$ or $(2+\sqrt2)\square$. Thus Lemma~\ref{le:indecStrengthen} applies.
			
			\item \label{pr:IndecAboveSqrt2-2}
			Let $Q_2$ be universal. We can rewrite it as
			\[
			Q_2(\iks,\ips,\iss)
			= \iks^2 + (2+\sqrt2)\bigl(\,\underbrace{\ips^2+(2-\sqrt2)\ips \iss+(1-\sqrt2)^2\iss^2}_{G_2(\ips,\iss)}\,\bigr).
			\]
			It is easy to see that the quadratic form $G_2$ is positive definite, i.e., $G_2(y,z)\in\OK^{+,0}$ for any $y,z\in\OK$. We thus obtain:
			\[
			\OK^{+,0} \subset \OKctv + (2+\sqrt2)\OK^{+,0}.
			\]
			Iteratively, we get
			\[
			\OK^{+,0} \subset \OKctv + (2+\sqrt2)\OKctv + (2+\sqrt2)^2\OK^{+,0},
			\]
			from which we see that every indecomposable element is either $\square$ or $(2+\sqrt2)\square$; no multiple of $(2+\sqrt2)^2$ can be indecomposable. Now we can once again apply Lemma~\ref{le:indecStrengthen}.

			\item \label{pr:IndecAboveSqrt2-3}
			Suppose that $Q_3$ is universal. We have
			\[
			Q_3(\iks,\ips,\iss)
			= \iks^2 + \iss^2 + (2+\sqrt2)\bigl(\,\underbrace{\ips^2+(2-\sqrt2)\ips \iss+(2-\sqrt2)\iss^2}_{G_3(\ips,\iss)}\,\bigr).
			\]
			The quadratic form $G_3$ is positive definite. Therefore,
			\[
			\OK^{+,0} \subset \OKctv + \OKctv + (2+\sqrt2)\OK^{+,0},
			\]
			and, iteratively, we get
			\[
			\OK^{+,0} \subset \OKctv +\OKctv + (2+\sqrt2)(\OKctv +\OKctv) + (2+\sqrt2)^2\OK^{+,0}.
			\]
			It again follows that every indecomposable element is either $\square$ or $(2+\sqrt2)\square$, so Lemma~\ref{le:indecStrengthen} can be applied.
			
			\item[(3$'$)] \label{pr:IndecAboveSqrt2-3'}
			Finally, let $Q_3'$ be universal. We have
			\[
			Q_3'(\iks,\ips,\iss)
			= \iks^2 + \iss^2 + (2+\sqrt2)\bigl(\,\underbrace{(3-2\sqrt2)\ips^2+(2-\sqrt2)\ips \iss+(2-\sqrt2)\iss^2}_{G_3'(\ips,\iss)}\,\bigr)
			\]
			with $G_3'$ positive definite. Thus, we proceed as in \ref{pr:IndecAboveSqrt2-3}. \qedhere
		\end{enumerate}
	\end{proof}
	
	The only field known to have the structure of indecomposables $\UKctv \cup (2+\sqrt2)\UKctv$ is $\Q(\!\sqrt2)$ itself. Also, the equality $\norm{K/\Q}{2+\sqrt2}=2^{[K:\Q]/2}$ yields the following corollary, which will serve as a useful criterion later.
	
	\begin{corollary} \label{co:NormIndecAboveSqrt2}
		Let $K\ni\sqrt2$ be a totally real number field with $[K:\Q]=d$. Assume that at least one of the quadratic forms in \eqref{eq:ListQFs2} is universal over $K$. Then all indecomposable elements in $K$ have norm $1$ or $2^{d/2}$.
	\end{corollary}

	\subsection{Maximal totally real subfields of cyclotomic fields}
	Let $k\in\Z^+$. We denote $\zeta_k^{}={\mathrm{e}}^{2\pi \mathrm{i}/k}$  and $F_k=\Q(\zeta_k^{}+\zeta_k^{-1})$,\footnote{Sometimes also denoted by $\Q(\zeta_k^{})^+$.} i.e., $F_k$ is the maximal totally real subfield of the $k$th cyclotomic field $\Q(\zeta_k^{})$. Note that $F_1=F_2=F_3=F_4=F_6=\Q$, that $F_5=\Q(\!\sqrt5)$ and $F_8=\Q(\!\sqrt2)$.
	
	\medskip
	The following proposition is a generalisation of \cite[Prop.~6.2 and~7.1]{KY-EvenBetter}, where it was proved in the cases when $k=2^n$ and $k=p$ is an odd prime.
	
	\begin{proposition}\label{pr:alphakProperties}
		For $k\geq3$, define
		\[
		\alpha_k=2+\zeta_k^{}+\zeta_k^{-1} \quad \text{and} \quad \beta_k=2-\zeta_k^{}-\zeta_k^{-1}.
		\]
		Then the following hold:
		\begin{enumerate}
			\item $\alpha_k,\beta_k\in\OKPlus[F_k]$; \label{pr:alphakProperties-totpos}
			\item $\norm{F_k/\Q}{\alpha_k}=\begin{cases}
				p & \text{if } k=2p^n \text{ (where $p$ is a prime and $n\geq1$),}\\
				1 & \text{otherwise;}
			\end{cases} 
			$
			
			\vspace{2pt}\noindent $\norm{F_k/\Q}{\beta_k}=\begin{cases}
				p & \text{if } k=p^n \text{ (where $p$ is a prime and $n\geq1$),}\\
				1 & \text{otherwise;}
			\end{cases}
			$
			\label{pr:alphakProperties-norm}
			\vspace{2pt}
			\item $\Tr{F_k/\Q}{\alpha_k}=2[F_k:\Q]+\mu(k)$ and $\Tr{F_k/\Q}{\beta_k}=2[F_k:\Q]-\mu(k)$, where $\mu(k)$ is the M\"obius function. \label{pr:alphakProperties-trace}
			\item If $[F_k:\Q]>2$, then both $\alpha_k$ and $\beta_k$ are indecomposable in $F_k$. \label{pr:alphakProperties-indec}
		\end{enumerate}
	\end{proposition}
	\begin{proof}
		\ref{pr:alphakProperties-totpos} Since $\OK[F_k]=\Z[\zeta_k^{}+\zeta_k^{-1}]$, we have $\alpha_k, \beta_k\in\OK[F_k]$. Hence, it suffices to prove that $-2<\zeta_k^j+\zeta_k^{-j}<2$ for all $1\leq j\leq k$ such that $\gcd(j,k)=1$. We have
		\[
		\zeta_k^j+\zeta_k^{-j}={\mathrm{e}}^{2\pi \mathrm{i} j/k}+{\mathrm{e}}^{-2\pi \mathrm{i} j/k}=2\cos\frac{2\pi j}{k};
		\]
		hence, it follows that
		\[
		-2\leq \zeta_k^j+\zeta_k^{-j}\leq 2.
		\]
		There is an equality on the left side if and only if $2j=k$ and on the right side if and only if $j=k$; however, both cases are impossible as $\gcd(j,k)=1$ and $k\geq 3$. The claim follows.
		
		\ref{pr:alphakProperties-norm} Denote $E_k=\Q(\zeta_k)$. First, we look at $\alpha_k$: we have
		\[
		\norm{E_k/F_k}{1+\zeta_k^{}}=(1+\zeta_k^{})(1+\zeta_k^{-1})=2+\zeta_k^{}+\zeta_k^{-1}=\alpha_k.
		\]
		On the other hand, we also have
		\[
		\norm{E_k/\Q}{1+\zeta_k^{}}=(-1)^{\varphi(k)}\Phi_k(-1),
		\]
		where $\varphi(k)$ is the Euler's totient function and $\Phi_k$ is the $k$th cyclotomic polynomial. Using \cite[Lemma~7]{BHPM-CyclotomicPoly} and the fact that $\varphi(k)$ is even for $k\geq3$, we get
		\[
		\norm{E_k/\Q}{1+\zeta_k^{}} = \begin{cases}
			p & \text{ if } k=2p^n \text{ ($p$ prime, $n\geq1$),}\\
			1 & \text{otherwise.}
		\end{cases}
		\]
		The claim follows.
		
		The proof for $\beta_k$ is analogous, using $\beta_k=\norm{E_k/F_k}{1-\zeta_k^{}}$ and $\norm{E_k/\Q}{1-\zeta_k}=\Phi_k(1)$, and applying \cite[Lemma~4]{BHPM-CyclotomicPoly}.
		
		\ref{pr:alphakProperties-trace} We have
		\[
		\Tr{E_k/\Q}{\zeta_k^{-1}}=\Tr{E_k/\Q}{\zeta_k^{}}=\sum_{\substack{1 \leq j \leq k,\\\gcd(j,k)=1}}\zeta_k^j=\mu(k).
		\]
		Thus,
		\[
		2\Tr{F_k/\Q}{\alpha_k}
		=\Tr{E_k/\Q}{\alpha_k}
		=2[E_k:\Q]+\Tr{E_k/\Q}{\zeta_k^{}}+\Tr{E_k/\Q}{\zeta_k^{-1}}
		=4[F_k:\Q]+2\mu(k).    
		\]
		The proof for $\beta_k$ is analogous.
		
		\ref{pr:alphakProperties-indec} Note that we always have $\Tr{F_k/\Q}{\alpha_k}\leq 2[F_k:\Q]+1$. Hence, assuming $[F_k:\Q]>2$, we get $\Tr{F_k/\Q}{\alpha_k}< \frac52[F_k:\Q]$. Therefore, $\alpha_k$ is indecomposable by Lemma~\ref{le:smallAreIndecomposable}. The proof for $\beta_k$ is analogous.
	\end{proof}

	\begin{lemma} \label{le:CyclotomicSubfield}
		Let $K\ni\sqrt{2}$ be a totally real number field, and assume that at least one of the following holds:
		\begin{enumerate}
			\item $F_p\subset K$ for some prime number $p>5$, or \label{le:CyclotomicSubfield-big}
			\item $F_{p^2}\subset K$ for $p \in\{3,5\}$, or 
			\label{le:CyclotomicSubfield-35}
			\item $F_{p^4}\subset K$ for $p=2$. 
			\label{le:CyclotomicSubfield-2}
		\end{enumerate}
		Then none of the quadratic forms in \eqref{eq:ListQFs2} is universal over $K$.
	\end{lemma}
	\begin{proof}
		Let $d=[K:\Q]$. We prove all three claims simultaneously: We fix the appropriate prime $p$ and let $n$ be the largest integer such that $F_{p^n}\subset K$. Thus, by assumption, $n\geq 4$ if $p=2$, and $n\geq 2$ if $p=3$ or $5$. Consider the element $\beta_{p^n}=2-\zeta_{p^n}^{}-\zeta_{p^n}^{-1}$. By Proposition~\ref{pr:alphakProperties}, we have $\beta_{p^n}\in\OKPlus[F_{p^n}]\subset\OKPlus[K]$, and $\norm{F_{p^n}/\Q}{\beta_{p^n}}=p$, and $\Tr{F_{p^n}/\Q}{\beta_{p^n}}=2[F_{p^n}:\Q]-\mu(p^n)$. Then
		\begin{align*}
			& \norm{K/\Q}{\beta_{p^n}}=\norm{F_{p^n}/\Q}{\norm{K/F_{p^n}}{\beta_{p^n}}}=p^{[K:F_{p^n}]},\\
			& \Tr{K/\Q}{\beta_{p^n}}=[K:F_{p^n}]\cdot\Tr{F_{p^n}/\Q}{\beta_{p^n}}=2d-\mu(p^n)[K:F_{p^n}].
		\end{align*}
		If $n\geq2$, then $\Tr{K/\Q}{\beta_{p^n}}=2d<\frac52d$. If $n=1$, then $\Tr{K/\Q}{\beta_{p}}=2d+[K:F_{p}]$, with
		\[
		[K:F_p]=\frac{[K:\Q]}{[F_p:\Q]}=\frac{d}{\frac{\varphi(p)}{2}}=\frac{2d}{p-1}<\frac{2d}{4}=\frac{d}{2},
		\]
		assuming $p>5$. In particular, we get $\Tr{K/\Q}{\beta_{p}}<\frac52d$ again. Thus, in either case, $\beta_{p^n}$ is indecomposable in $K$ by Lemma~\ref{le:smallAreIndecomposable}. 
		
		Finally, we prove that $\norm{K/\Q}{\beta_{p^n}}$ is neither $1$ nor $2^{d/2}$. This is clear if $p$ is odd. For $p=2$, note that $F_{2^n} \supsetneq F_8=\Q(\!\sqrt2)$, so $[K:F_{2^n}]<[K:\Q(\!\sqrt2)]=\frac d2$ and thus
		\[
		\norm{K/\Q}{\beta_{2^n}}=2^{[K:F_{2^n}]}<2^{d/2}.
		\]
		In all cases, $\beta_{p^n}$ is an indecomposable element of norm other than $1$ and $2^{d/2}$. Therefore, the claim follows from Corollary~\ref{co:NormIndecAboveSqrt2}.
	\end{proof}
	
	\subsection{Lifted quadratic forms are not universal}
	
	We can now prove the main theorem of this section. Note that part \ref{th:QFNotUniversal-cyclotomic} is precisely Theorem~\ref{th:IntroQFNotUniversal}.
	
	\begin{theorem} \label{th:QFNotUniversal}
		Let $K\ni\sqrt{2}$ be a totally real number field. Assume that at least one of the following holds:
		\begin{enumerate}
			\item $\abs{\UKPlus/\UKctv}>1$, or  \label{th:QFNotUniversal-units}
			\item $2\OKPlus\not\subset\sum^4\square$, or  \label{th:QFNotUniversal-2OK}
			\item there exists an indecomposable $\alpha\in\OKPlus$ such that $\alpha\notin\UKctv\cup(2+\sqrt2)\UKctv$, or
			\label{th:QFNotUniversal-indec}
			\item $K$ contains the field $F_k=\Q(\zeta_k^{}+\zeta_k^{-1})$ for some $k \geq 1$, different from $\Q$ and $\Q(\!\sqrt2)$.\label{th:QFNotUniversal-cyclotomic}
		\end{enumerate}
		Then there is no ternary classical quadratic form that is universal both over $\Q(\!\sqrt2)$ and over $K$. In other words, none of the quadratic forms in \eqref{eq:ListQFs2} is universal over $K$.
	\end{theorem}
	\begin{proof}
		First, note that every ternary classical quadratic form that is universal over $\Q(\!\sqrt2)$ is isometric to one of the quadratic forms in \eqref{eq:ListQFs2}. Therefore, it is enough to prove their non-universality over $K$.
		
		\medskip
		
		\ref{th:QFNotUniversal-units} No ternary classical quadratic form is universal over such field by Proposition~\ref{pr:KKK}.
		
		\medskip
		
		\ref{th:QFNotUniversal-2OK} Let $Q$ be one of the quadratic forms in \eqref{eq:ListQFs2}. By Lemma~\ref{le:2Q}, the quadratic form $2Q$ only represents elements of $\sum^4\square$. Using the assumption, it follows that there exists an element of $2\OKPlus$ that is not represented by $2Q$. Therefore, the quadratic form $Q$ is not universal.
		
		\medskip
		
		\ref{th:QFNotUniversal-indec} This follows directly from Proposition~\ref{pr:IndecAboveSqrt2}.
		
		\medskip
		
		\ref{th:QFNotUniversal-cyclotomic} Let $[K:\Q]=d$. Let $k$ be such that $F_k\subset K$ and $F_k$ is not $\Q$ or $\Q(\!\sqrt2)$. Let $R$ be the field of all totally real algebraic numbers; then $F_k=\Q(\zeta_k^{})\cap R$. Let $p$ be the largest prime dividing $k$. It clearly holds that $\Q(\zeta_p^{})\subset\Q(\zeta_k)$.  It follows that $\Q(\zeta_p^{})\cap R\subset \Q(\zeta_k^{})\cap R$, i.e., $F_p\subset F_k$. Thus $F_p\subset K$.
		
		If $p>5$, then the claim follows directly from Lemma~\ref{le:CyclotomicSubfield}\ref{le:CyclotomicSubfield-big}. 
		
		For $p=5$, note that $F_5=\Q(\!\sqrt5)$, so we have $\Q(\!\sqrt2,\sqrt5)\subset K$. Let $\gamma=\frac{5+\sqrt2+\sqrt5+\sqrt{10}}2$. Then $\gamma\in\OKPlus$ and $\norm{K/\Q}{\gamma}=3^{d/2}$. Since $3^{d/2}<2^d$, we get that $\gamma$ is indecomposable in $K$ by Lemma~\ref{le:smallAreIndecomposable}. Moreover, the norm of $\gamma$ is clearly different from $1$ and $2^{d/2}$; thus, the claim follows from Corollary~\ref{co:NormIndecAboveSqrt2}.
		
		Suppose $p \leq 3$. If $3^2\mid k$ or $2^4\mid k$, then we use part \ref{le:CyclotomicSubfield-35} or part \ref{le:CyclotomicSubfield-2} of Lemma~\ref{le:CyclotomicSubfield}, respectively. Now the only possibilities left are $k=2^a3^b$ with $0\leq a \leq 3$ and $0\leq b\leq1$. However, $F_1=F_2=F_3=F_4=F_6=\Q$ and $F_8=\Q(\!\sqrt2)$, so the only options are $F_{12}$ and $F_{24}$. In either case, $F_{12}\subset K$. In $F_{12}=\Q(\!\sqrt3)$, there is $2+\sqrt3 \in \UKPlus[F_{12}] \setminus \UKctv[F_{12}]$, so we have $\abs{\UKPlus/\UKctv} \geq \abs{\UKPlus[F_{12}]/\UKctv[F_{12}]}\geq 2$ by Lemma~\ref{le:unitsDontDissappear}. Thus we can use part \ref{th:QFNotUniversal-units}.
	\end{proof}
	
	The following corollary provides four conditions \ref{it:NotCriterion-a}--\ref{it:NotCriterion-d}; if any of them is not satisfied, then $\OKPlus[\Q(\!\sqrt2)]$ is not a criterion set for $K$. This strongly indicates that it is not a criterion set for any proper extension of $\Q(\!\sqrt2)$.
	
	\begin{corollary} \label{co:NotCriterion}
		Let $K\ni\sqrt2$ be a totally real number field such that $\OKPlus[\Q(\!\sqrt2)]$ is a criterion set for $K$. Then:
		\begin{enumerate}
			\item $\abs{\UKPlus/\UKctv}=1$. \label{it:NotCriterion-a}
			\item $2\OKPlus\subset\sum^3\square$. \label{it:NotCriterion-b}
			\item The set of indecomposables in $K$ is $\UKctv\cup(2+\sqrt2)\UKctv$. \label{it:NotCriterion-c}
			\item If $F_k=\Q(\zeta_k^{}+\zeta_k^{-1}) \subset K$, then $F_k=\Q$ or $\Q(\!\sqrt2)$. \label{it:NotCriterion-d}
		\end{enumerate}
	\end{corollary}
	\begin{proof}
		Consider $Q_2$ from \eqref{eq:ListQFs2}. Since $Q_2$ is universal over $\Q(\!\sqrt2)$, it represents all of $\OKPlus[\Q(\!\sqrt2)]$ over $\Q(\!\sqrt2)$, and hence also over $K$. Thus, by assumption, it is universal over $K$. This means that none of the conditions of Theorem~\ref{th:QFNotUniversal} can be satisfied; this directly yields \ref{it:NotCriterion-a}, \ref{it:NotCriterion-c}, \ref{it:NotCriterion-d} and a weaker version of \ref{it:NotCriterion-b}, namely $2\OKPlus\subset\sum^4\square$. The full statement \ref{it:NotCriterion-b} follows from the universality of $Q_2$ using Remark~\ref{re:2Q2}.
	\end{proof}
	
	We would like to point out that the quadratic forms from \eqref{eq:ListQFs2} are in particular not universal over any quartic field; this fact will be proven in Lemma~\ref{le:noLift}.

	\section{Kitaoka's conjecture over quartic fields}
	
	We devote this section to the study of universal ternary classical free lattices over quartic fields; our goal is to prove Theorem~\ref{th:QuarticKitaoka2}, i.e., there are no such lattices. Thanks to \cite[Thm.~1.2(A)]{KKK}, it is enough to prove the result for fields containing $\sqrt2$. We will distinguish between two types of lattices -- the ones that contain a ternary sublattice that is defined and universal over $\Q(\!\sqrt2)$, and the ones that do not contain such a sublattice. 
	
	Recall that by \cite{CKR}, the universal ternary classical lattices over $\Q(\!\sqrt2)$ (which will need to be analysed separately) correspond to the four quadratic forms in \eqref{eq:ListQFs2}; we repeat them here for the reader's convenience:
	\begin{equation}\begin{array}{l} \label{eq:ListQFs}
			Q_1(\iks,\ips,\iss)=\iks^2+\ips^2+(2+\sqrt2)\iss^2,\\ [2mm]
			Q_2(\iks,\ips,\iss)=\iks^2+(2+\sqrt2)\ips^2+2\ips\iss+(2-\sqrt2)\iss^2,\\[2mm]
			Q_3(\iks,\ips,\iss)=\iks^2+(2+\sqrt2)\ips^2+2\ips\iss+3\iss^2,\\[2mm]
			Q_3'(\iks,\ips,\iss)=\iks^2+(2-\sqrt2)\ips^2+2\ips\iss+3\iss^2.
	\end{array}\end{equation}
	The corresponding lattices are then 
	\begin{equation}\begin{array}{lll}\label{eq:ListLattices}
			L_1\simeq\qf{1,1,2+\sqrt2}, &\quad &L_2\simeq\qf{1}\perp \anglematrix{2+\sqrt2}1{2-\sqrt2},\\ [2mm] L_3\simeq\qf{1}\perp \anglematrix{2+\sqrt2}1{3}, &\quad &L_3'\simeq\qf{1}\perp \anglematrix{2-\sqrt2}1{3}.
	\end{array}\end{equation}
	
	If $\sqrt2\in K$, one can consider these lattices over $K$ instead of over $\Q(\!\sqrt2)$. Formally, they can be defined as $\OK\otimes L_i$, but since \eqref{eq:ListLattices} does not specify the base field, we typically just write $L_i$ and the base field is clear from context.
	
	We introduce a special terminology for overlattices of $\OK \otimes L$ where $L$ is isometric to one of the above lattices.
	
	\begin{definition}
		Let $K\ni\sqrt2$ be a totally real number field. We say that a ternary classical lattice over $K$ is of \emph{type LIRE} (standing for \enquote{lifted and refined}) if it contains a sublattice of the form $\OK \otimes L$ where $L$ is a universal ternary classical lattice over $\Q(\!\sqrt2)$. Otherwise, we say that it is of \emph{type non-LIRE}.
	\end{definition}

	In Subsection~\ref{ss:LIRE}, we study lattices of type LIRE; the main result is Proposition~\ref{pr:noLIRE}. Then, in Subsection~\ref{ss:nonLIRE}, we deal with the lattices of type non-LIRE, with the results summarised in Proposition~\ref{pr:nonLIRE-5lambda}. We leave out five fields that require special treatment, which we provide in Subsection~\ref{ss:Exceptional}.
	
	\subsection{Lattices of type LIRE}\label{ss:LIRE}
	The goal of this subsection is to prove that if $K\ni\sqrt2$ is a quartic field and $L$ is one of the lattices in \eqref{eq:ListLattices} considered as a lattice over $K$, then no classical free overlattice of $L$ is universal. We start with the simplest case -- as an immediate corollary of results from Section~\ref{se:lifting}, we prove that none of the corresponding quadratic forms is universal over any field of degree 4. Then, separately for each lattice from \eqref{eq:ListLattices}, we will examine the more difficult question of whether it can have a proper universal classical overlattice.
	
	\begin{lemma} \label{le:noLift}
		Let $K\ni\sqrt2$ be a totally real number field with $[K:\Q]=4$. Let $Q$ be one of the quadratic forms in \eqref{eq:ListQFs} that are universal over $\Q(\!\sqrt2)$. Then $Q$ is not universal over $K$.
	\end{lemma}
	\begin{proof}
		By \cite[Thm.~3.3]{KY-EvenBetter}, no quartic fields except for $\Q(\!\sqrt2,\sqrt5)$ and $\Q(\zeta_{20}^{}+\zeta_{20}^{-1})$ satisfy $2\OKPlus\subset\sum\square$, so the claim follows from Theorem~\ref{th:QFNotUniversal}\ref{th:QFNotUniversal-2OK}. The field $\Q(\zeta_{20}^{}+\zeta_{20}^{-1})$ does not contain $\sqrt2$. The field $\Q(\!\sqrt2,\sqrt5)$ is biquadratic and these do not admit universal ternary quadratic forms by \cite{KTZ}; alternatively, it contains $F_5=\Q(\zeta_5^{}+\zeta_5^{-1})=\Q(\!\sqrt5)$, so the non-universality also follows from Theorem~\ref{th:QFNotUniversal}\ref{th:QFNotUniversal-cyclotomic}.
	\end{proof}

	\subsubsection{Lattices containing \texorpdfstring{${L_1}$}{L_1}} \label{ss:L1}
	
	Here we ask whether there can exist a universal ternary classical free lattice over $K$ containing $L_1\simeq\qf{1,1,2+\sqrt2}$ as a sublattice. 
	
	\begin{lemma}\label{le:nutneI3}
		Let $K \ni \sqrt2$ be a totally real number field of degree $4$ with $\UKPlus=\UKctv$. Let $L$ be a classical free lattice over $K$ containing $L_1\simeq\qf{1,1,2+\sqrt2}$ as a sublattice. Then $L\simeq\qf{1,1,2+\sqrt2}$ or $L\simeq\qf{1,1,1}$. 
	\end{lemma}
	\begin{proof}
		We claim that every classical free overlattice of $L_1$ is of the form $L\simeq\qf{1,1,\alpha}$ for some $\alpha\in\OK$, where $2+\sqrt2 = \alpha t^2$ for some $t\in\OK$. One can either deduce this using $\dual{L_1}$ as in Subsection~\ref{ss:L3}, or see it easily as follows:
		
		It is in general true that any classical lattice containing $\qf{1,1}$ is of the form $\qf{1,1}\perp L_0$ where all elements of $L_0$ are orthogonal to $\qf{1,1}$. We now apply this statement both to $L_1$ and to $L$, both of which contain $\OK\e_1+\OK\e_2 \simeq \qf{1,1}$; here, $\e_1,\e_2,\f$ is the orthogonal basis of $L_1$ where $Q(\e_i)=1$ and $Q(\f)=2+\sqrt2$. Clearly, since $KL=KL_1$ is a ternary quadratic space, a vector from $KL_1$ is orthogonal to $\OK\e_1+\OK\e_2$ if and only if it is a $K$-multiple of $\f$. Thus $L=\OK\e_1+\OK\e_2+\OK\vv$ for some vector $\vv \in K\f$. Moreover, since $L_1\subset L$, we must have $\f\in L$, so $\f\in\OK\vv$. If we denote $\alpha=Q(\vv)$, we get $L\simeq \qf{1,1,\alpha}$ where $2+\sqrt2=Q(\f)=Q(t\vv)=t^2\alpha$ for some $t\in\OK$.
		
		Now we just apply this general description of classical free overlattices of $L_1$ to quartic fields with $\UKPlus=\UKctv$: If $t \in \UK$, then $\qf{1,1,\alpha}\simeq \qf{1,1,2+\sqrt2}\simeq L_1$. If $t \notin \UK$, then $\norm{K/\Q}{t}= \pm 2$ and $\alpha$ is a unit. Since $\UKPlus=\UKctv$, we have $\alpha\in\UKctv$ and therefore $\qf{1,1,\alpha}\simeq\qf{1,1,1}$.
	\end{proof}

	\begin{proposition} \label{pr:L1}
		Let $K\ni\sqrt2$ be a totally real number field with $[K:\Q]=4$. Let $L$ be a ternary classical free lattice over $K$ containing $L_1$ as a sublattice. Then $L$ is not universal.
	\end{proposition}
	\begin{proof}
		It suffices to consider the case $\UKPlus=\UKctv$, otherwise there can be no universal ternary classical lattice over $K$ containing $\sqrt2$ by Proposition~\ref{pr:KKK}. Therefore, we can apply Lemma~\ref{le:nutneI3} to get $L\simeq\qf{1,1,1}$ or $L\simeq\qf{1,1,2+\sqrt2}$. However, the former is not universal by \cite{Si} and the latter by Lemma~\ref{le:noLift}.
	\end{proof}
	
	
	\subsubsection{Lattices containing \texorpdfstring{${L_2}$}{L_2}} \label{ss:L2}
	
	Recall that $L_2\simeq\qf{1}\perp \anglematrix{2+\sqrt2}1{2-\sqrt2}$. Note that $L_2$ is a classical free lattice with a unit determinant, i.e., it is unimodular.
	
	\begin{proposition} \label{pr:L2}
		Let $K\ni\sqrt2$ be a totally real number field with $[K:\Q]=4$. Let $L$ be a ternary classical lattice over $K$ containing $L_2$ as a sublattice. Then $L$ is not universal.
	\end{proposition}
	\begin{proof}
		Due to unimodularity, $L_2$ has no proper classical overlattice, so in fact $L = L_2$; however, $L_2$ is not universal over $K$ according to Lemma~\ref{le:noLift}.
	\end{proof}
	
	\subsubsection{Lattices containing \texorpdfstring{${L_3}$}{L_3}} \label{ss:L3}
	
	Recall that $L_3\simeq\qf{1}\perp \anglematrix{2+\sqrt2}1{3}$. This is the last remaining lattice to be dealt with, as examining the lattice $L_3'\simeq\qf{1}\perp \anglematrix{2-\sqrt2}1{3}$ is unnecessary thanks to the symmetry.

	The direct analogy of Lemma~\ref{le:nutneI3} does not hold in this case. It is still true that every classical lattice over a quartic field containing $L_3$ as a sublattice is either $L_3$, or unimodular -- but this time we do not know that the unimodular lattice represents $\qf{1,1}$, as in the following example. In particular, it could happen that the lattice is not diagonalisable.
	
	\begin{example} \label{ex:exceptionalField}
		Let $K=\Q\bigl(\!\sqrt{3-\sqrt2}\bigr)$ and $\alpha=(1+\sqrt2)\sqrt{3-\sqrt2}=\sqrt{5+3\sqrt2}$.
		
		Set $L = \qf{1} \perp \anglematrix{3}{\alpha}{2+\sqrt2}$. Note that $L$ is unimodular, since $\det L = 1$. However, $L$ is not diagonalisable: this is easy to see from the fact that the binary part does not represent any unit at all (which can comfortably be verified by a computation, e.g., in Magma). We claim that $L$ contains a sublattice isometric to $L_3$.
		
		We exhibit an explicit sublattice of $\anglematrix{3}{\alpha}{2+\sqrt2}$ isometric to $\anglematrix{2+\sqrt2}1{3}$: let $\e$ and $\f$ be the vectors such that $Q(\e)=3$, $Q(\f)=2+\sqrt2$, $B_Q(\e,\f)=\alpha$. Then, for the vectors $\vx_1 = \f$ and $\vx_2 =-(1+\sqrt2)\sqrt{3-\sqrt2}\,\e + 3\f$, we indeed have $Q(\vx_1)=2+\sqrt2$, $Q(\vx_2)=3$, $B_Q(\vx_1,\vx_2)=1$.
		It follows that $L_3$ is isometric to a sublattice of $L$.
		
		The lattice $L$ is not universal, and in fact no universal ternary classical lattice can exist over this field by Proposition~\ref{pr:KKK} as $K$ contains both $\sqrt2$ and a nonsquare totally positive unit $\ve = -3\vartheta^3 + 6\vartheta^2 + 5\vartheta - 9$, where $\vartheta=\sqrt{3-\sqrt2}$. 
	\end{example}
	
	\medskip
	
	Our approach is to use the fact that all classical overlattices of $L_3$ are sublattices of its dual $\dual{L}_3$ (see Lemma~\ref{le:DualLattice}). Denote $L_0=\anglematrix{2+\sqrt2}1{3}$; then $L_3=\qf{1}\perp L_0$. Using that $\dual{(J_1\perp J_2)}\simeq\dual{J_1}\perp\dual{J_2}$ for any lattices $J_1$ and $J_2$ (see \cite[\S82.F]{OMbook}) and that $\dual{\qf{1}}=\qf{1}$, we get that $\dual{L_3}\simeq\qf{1}\perp\dual{L_0}$. In particular, we only need to compute $\dual{L_0}$. 
	
	For the rest of this subsection, we denote $(\e_1,\e_2)$ the basis of $L_0$ with respect to which the Gram matrix of $L_0$ is $G=\bigl(\begin{smallmatrix}2+\sqrt2&1\\1&3\end{smallmatrix}\bigr)$. Furthermore, let $(\f_1,\f_2)$ be the basis of $\dual{L_0}$ with respect to which the Gram matrix of $\dual{L_0}$ is $G^{-1}$.
	By a direct application of Lemma~\ref{le:DualGramMatrix}, we get the following lemma.
	
	\begin{lemma} \label{le:DualBasis}
		Let $\e_1,\e_2,\f_1,\f_2$ be as above. Then
		\begin{equation} \label{eq:DualBasis}
			\f_1=\frac{1}{5+3\sqrt2}(3\e_1-\e_2),\quad \f_2=\frac{1}{5+3\sqrt2}\bigl(-\e_1+(2+\sqrt2)\e_2\bigr).
		\end{equation}
		In particular, $\dual{L_0}\simeq\qf{H}$, where
		\[
		H=\frac{1}{5+3\sqrt2}\begin{pmatrix} 3 & -1 \\ -1 & 2+\sqrt2 \end{pmatrix}.
		\]
	\end{lemma}
	
	We have $L_0=\OK\e_1+\OK\e_2$ and $\dual{L_0}=\OK\f_1+\OK\f_2$. In the following two lemmas, we will find alternative pseudobases of $L_0$ and $\dual{L_0}$ that will be easier to compare. We put $p_7=5+3\sqrt2$ to increase readability.
	
	\begin{lemma} \label{le:L0AltGenerators}
		Let $\e_1,\e_2,\f_1,\f_2$ be as above. Then $L_0=\OK p_7\f_1+\OK\e_1$.
	\end{lemma}
	\begin{proof}
		Denote $L=\OK p_7\f_1+\OK\e_1$. We want to prove $L=L_0$.
		
		Clearly, $\e_1\in L$. As $\e_2=3\e_1-p_7\f_1$, it follows that also $\e_2\in L$. Thus, $L_0\subset L$.
		
		On the other hand, $p_7\f_1=(3\e_1-\e_2)\in L_0$ and $\e_1\in L_0$; therefore, $L\subset L_0$. 
	\end{proof}
	
	\begin{lemma} \label{le:L0DualAltGenerators}
		Let $\e_1,\e_2,\f_1,\f_2$ be as above. Then $\dual{L}_0=\OK\f_1+\OK\e_1$.
	\end{lemma}
	\begin{proof}
		As the first step, we claim that $\f_2-2\f_1\in L_0$: denoting $\overline{p}_7=5-3\sqrt2$ (so that $p_7\overline{p}_7=7$), we have
		\begin{multline*}
			p_7(\f_2-2\f_1)
			=\bigl(-\e_1+(2+\sqrt2)\e_2\bigr)-2(3\e_1-\e_2)
			=-7\e_1+(4+\sqrt2)\e_2\\
			= p_7(-\overline{p}_7)\e_1+p_7(2-\sqrt2)\e_2
			=p_7\bigl(-\overline{p}_7\e_1+(2-\sqrt2)\e_2\bigr)\in p_7 L_0.
		\end{multline*}
		It follows that $\dual{L_0}=\OK\f_1+\OK\f_2\subset\OK\f_1+L_0$. On the other hand, we have $L_0\subset\dual{L_0}$ by Lemma~\ref{le:DualLattice} and clearly $\OK\f_1\subset\dual{L_0}$. Using Lemma~\ref{le:L0AltGenerators}, we thus get
		\[
		\dual{L_0}=\OK\f_1+L_0=\OK\f_1+\OK p_7\f_1+\OK\e_1=\OK\f_1+\OK\e_1.\qedhere
		\]
	\end{proof}
	
	Using Lemmas~\ref{le:L0AltGenerators} and~\ref{le:L0DualAltGenerators}, we can now easily describe all overlattices of $L_0$.
	
	\begin{proposition} \label{pr:L0Summary}
		Let $\e_1,\e_2,\f_1,\f_2$ be as above. Let $K\ni\sqrt2$ be a totally real number field. Furthermore, denote $p_7=5+3\sqrt2$. Then the lattices $L'$ satisfying $L_0 \subset L' \subset \dual{L_0}$ are precisely the lattices of the form
		\[
		L' = \ia \f_1 + \OK\e_1 \quad \text{with an ideal $\ia$ satisfying } p_7\OK \subset \ia \subset \OK.
		\]
		In particular, $L_0$ corresponds to the choice $\ia=p_7\OK$ and $\dual{L_0}$ to the choice $\ia=\OK$. Furthermore:
		\begin{enumerate}
			\item $L'$ is free if and only if $\ia$ is a principal ideal; \label{pr:L0Summary-free}
			\item $L'$ is classical if and only if $\ia^2 \subset p_7\OK$.\label{pr:L0Summary-classical}
		\end{enumerate}
	\end{proposition}
	\begin{proof}
		First, recall that $L_0=\OK p_7\f_1+\OK\e_1$ by Lemma~\ref{le:L0AltGenerators} and $\dual{L_0}=\OK\f_1+\OK\e_1$ by Lemma~\ref{le:L0DualAltGenerators}. It follows that any lattice $L'$ satisfying $L_0\subset L'\subset\dual{L_0}$ must be of the form $L'=\ia\f_1+\OK\e_1$ where $\ia$ is an ideal satisfying $p_7\OK\subset\ia\subset\OK$.
		
		Let $L'=\ia\f_1+\OK\e_1$ be one such lattice. 
		
		To prove \ref{pr:L0Summary-free}, recall that $(\f_1,\e_1)$ is a pseudobasis of $L'$. Thus, $L'$ is free if and only if $\ia$ is in the same class as $\OK$ in the class group of $K$, i.e., if and only if $\ia$ is a principal ideal. 
		
		As for \ref{pr:L0Summary-classical}, we simply need to verify for each pair of elements $\vv_1,\vv_2 \in L'$ whether $B_Q(\vv_1,\vv_2) \in \OK$. Let $\vv_i=\alpha_i\f_1+\beta_i\e_1$ with $\alpha_i\in\ia$ and $\beta_i\in\OK$ for $i=1,2$. As $L_0\subset L'\subset\dual{L_0}$, we can skip most of the verification: we already know that $B_Q(\vu,\vw)\in\OK$ for all $\vu\in L'$ and $\vw\in L_0$.  Hence, after expanding the expression $B_Q(\vv_1,\vv_2)=B_Q(\alpha_1\f_1+\beta_1\e_1,\alpha_2\f_1+\beta_2\e_1)$, we see that the lattice $L'$ is classical if and only if $B_Q(\alpha_1\f_1,\alpha_2\f_1) \in \OK$ for every $\alpha_1,\alpha_2\in\mathfrak{a}$. Recall that $Q(\f_1)=\frac{3}{p_7}$ by Lemma~\ref{le:DualBasis}; hence, $L'$ is classical if and only if
		\[
		\alpha_1\alpha_2 \frac{3}{p_7} \in \OK \quad \text{ for every } \alpha_1,\alpha_2\in\ia.
		\]
		When we sum this condition over all elements $\alpha_1,\alpha_2$, we obtain (from the definition of the product of ideals) the equivalent condition:
		\[
		\ia^2 \frac{3}{p_7} \subset \OK.
		\]
		Finally, since $3$ and $p_7$ are relatively prime (they have relatively prime norms), $3$ cannot play any role. The given condition is therefore indeed equivalent to
		\[
		\ia^2 \subset p_7\OK. \qedhere
		\]
	\end{proof}
	
	\begin{corollary} \label{co:summary}
		Let $K\ni\sqrt2$ be a totally real number field. The lattice $L_3$ has a proper classical free overlattice over $K$ if and only if $p_7=5+3\sqrt2$ is not squarefree.
	\end{corollary}
	\begin{proof}
		Instead of $L_3 = \qf{1} \perp L_0$, it suffices to consider only $L_0$; indeed, the classical overlattices of $L_3$ are precisely the lattices of the form $\qf{1} \perp L'$, where $L'$ is a classical overlattice of $L_0$.
		
		Thanks to  Lemma~\ref{le:DualLattice}, all classical overlattices of $L_0$ are contained in $\dual{L_0}$; all such lattices are described in Proposition~\ref{pr:L0Summary}. In particular, it follows that the classical free overlattices of $L_0$ correspond to the principal ideals $\mathfrak{a}$ satisfying $\mathfrak{a}^2 \subset p_7\OK \subset \mathfrak{a}$, i.e., to numbers $\alpha\in\OK$ satisfying $\alpha \mid p_7 \mid \alpha^2$; a proper sublattice corresponds to any non-trivial choice, i.e., $\alpha \notin p_7\UK$.
		
		These conditions can be rewritten as $p_7=\alpha t$ for $t \notin \UK$ together with $\alpha t \mid \alpha^2$, i.e., $t \mid \alpha$. Therefore, $\alpha=ts$ for some $s\in\OK$; hence, $p_7=t^2s$.
	\end{proof}
	
	We can finally prove the main result about overlattices of $L_3$.
	
	\begin{proposition} \label{pr:L3noFreeOverlattice}
		Let $K\ni\sqrt2$ be a quartic field in which $\UKPlus=\UKctv$. Then $L_3$ does not have a proper classical free overlattice over $K$. 
	\end{proposition}
	\begin{proof}
		By Corollary~\ref{co:summary}, we know that for a general field $K$, the lattice $L_3$ has a proper classical free overlattice if and only if $p_7=5+3\sqrt2$ is not squarefree. It remains to apply this to the case of a quartic field with units of all signatures.
		
		Note that $p_7$ is a prime in $\Q(\!\sqrt2)$. Since $K/\Q(\!\sqrt2)$ is an extension of degree $2$, the algebraic integer $p_7$ can only be divisible by a square if $p_7\OK=\ip^2$ and $\ip=\pi\OK$ is a principal ideal. This means that
		\[
		p_7 = \ve\pi^2 \qquad \text{ for } \ve \in \UK.
		\]
		Since both $p_7$ and $\pi^2$ are totally positive, we obtain that $\ve \in \UKPlus$. By assumption, $\ve=\eta^2$ for $\eta \in \UK$. Therefore, $p_7 = (\eta \pi)^2$. The field $K$ therefore necessarily contains $\sqrt{p_7}$. And since $p_7=5+3\sqrt2=(3-\sqrt2)(1+\sqrt2)^2$, it also contains $\sqrt{3-\sqrt2}$. However, for a quartic field, this forces $K = \Q\bigl(\!\sqrt{3-\sqrt2}\bigr)$. As we have seen in Example~\ref{ex:exceptionalField}, this field contains a nonsquare totally positive unit, and hence does not fulfil the assumptions.
	\end{proof}
	
	Note that the previous proposition talks only about \emph{free} lattices. A \emph{non-free} overlattice can easily exist -- in fact, it exists whenever $3-\sqrt2$ ramifies. 
	
	Now, as a corollary of Proposition~\ref{pr:L3noFreeOverlattice}, we get a complete result for all classical free lattices over quartic fields containing $L_3$.
	
	\begin{proposition} \label{pr:L3}
		Let $K\ni\sqrt2$ be a totally real number field with $[K:\Q]=4$. Let $L$ be a ternary classical free lattice over $K$ containing $L_3$ as a sublattice. Then $L$ is not universal.
	\end{proposition}
	\begin{proof}
		Thanks to Proposition~\ref{pr:KKK}, we can assume $\UKPlus=\UKctv$. Then, by Proposition~\ref{pr:L3noFreeOverlattice}, we see that the only option is $L=L_3$. However, $L_3$ is not universal according to Lemma~\ref{le:noLift}. 
	\end{proof}
	
	\subsubsection{Summary for lattices of type LIRE}
	Now we combine all the results for lattices of type LIRE.
	
	\begin{proposition} \label{pr:noLIRE}
		Let $K\ni \sqrt2$ be a totally real quartic field. Then there is no universal ternary classical free lattice of type LIRE over $K$.
	\end{proposition}
	\begin{proof}
		Let $L$ be a universal ternary classical free lattice of type LIRE over $K$. First, note that $L$ must have a sublattice $L'$ isometric to one of the lattices listed in \eqref{eq:ListLattices}; without loss of generality, assume that $L'$ is equal to one of these lattices. But then $L$ cannot be universal: for $L_1$, this is contained in Proposition~\ref{pr:L1}, for $L_2$ in Proposition~\ref{pr:L2} and for $L_3$ in Proposition~\ref{pr:L3}. For $L_3'$, this follows from symmetry -- if $L_3'$ had some universal classical free overlattice over some field $K$, we could easily find a field (in fact just another embedding of $K$) over which $L_3$ would also have a universal classical free overlattice, which would be a contradiction.
	\end{proof}
	
	\subsection{Lattices of type non-LIRE} \label{ss:nonLIRE}
	
	Our main strategy in this subsection is selecting four particular elements and showing that any non-LIRE classical lattice that represents them has rank at least $4$. The question is which four elements should we choose. Inspired by the results from Krásenský--Romeo \cite{KR}, we consider the elements $1, 2+\sqrt2, 3,3(2+\sqrt2)$, which form something we may call a \emph{ternary criterion} -- every classical lattice over $\Q(\!\sqrt2)$ that represents all of them is either universal ternary, or at least quaternary. For technical reasons, it will be useful to replace the last element by its $\UKctv$-multiple $3(2-\sqrt2)$, i.e., we will use the following four elements:
	\[
	1, 2+\sqrt2, 3,3(2-\sqrt2).
	\]
	Although we are interested in lattices not over $\Q(\!\sqrt2)$ but over its extensions, these four elements still prove to be a good choice for most fields. In particular, we will distinguish two cases depending on whether there exists a certain kind of \enquote{small} elements in $K$ or not -- and we will use the above-listed four elements in the latter case.
	
	\subsubsection{When there are no small elements}
	Here we study the fields where all integral elements $\omega$ satisfying $\omega^2 \preceq 6$ or $\omega^2 \preceq 3(2+\sqrt2)$ lie in $\Q(\!\sqrt2)$; this condition will later be denoted by \eqref{eq:SmallCondition}. We do not restrict to quartic fields; the main result, Proposition~\ref{pr:nonLIRE-5lambda}, holds for fields of all degrees.

	We start by listing some \enquote{small squares} in $\Q(\!\sqrt2)$.
	
	\begin{lemma} \label{le:smallsquares}
		Let $\omega \in \OK[\Q(\!\sqrt2)] = \Z[\sqrt2]$. Then:
		\begin{enumerate}
			\item If $\omega^2 \preceq 2+\sqrt2$, then $\omega = 0$. \label{le:smallsquares-a}
			\item If $\omega^2 \preceq 3$, then $\omega \in \{0,\pm1,\pm\sqrt2\}$. \label{le:smallsquares-b}
			\item If $\omega^2 \preceq 6$, then $\omega \in \{0, \pm1, \pm2, \pm\sqrt2, \pm(1+\sqrt2),\pm(1-\sqrt2)\}$. \label{le:smallsquares-c}
			\item If $\omega^2 \preceq 3(2+\sqrt2)$, then $\omega \in \{0, \pm1, \pm(1+\sqrt2)\}$. \label{le:smallsquares-d}
			\item If $\omega^2 \preceq 3(2-\sqrt2)$, then $\omega \in \{0, \pm1, \pm(1-\sqrt2)\}$. \label{le:smallsquares-e}
		\end{enumerate}
	\end{lemma}
	\begin{proof}
		This is a short and straightforward computation. Also, note that \ref{le:smallsquares-a} follows directly from the fact that $2+\sqrt2$ is a nonsquare indecomposable, that \ref{le:smallsquares-b} is a corollary of \ref{le:smallsquares-c}, and that \ref{le:smallsquares-d} and \ref{le:smallsquares-e} are equivalent.
	\end{proof}
	
	For the rest of the section, we use the notation $\lambda=2+\sqrt2$.
	
	\begin{lemma} \label{le:3x3}
		Let $K\ni \sqrt2$ be a totally real number field in which the following condition is satisfied for every $\omega\in\OK$:
		\begin{equation}\label{eq:SmallConditionWeak}
			\omega^2\preceq 3\lambda\ \text{ or }\ \omega^2 \preceq 3 \quad \Longrightarrow \quad \omega \in \Q(\!\sqrt2).
		\end{equation}
		Let $L = \OK \vv_1 + \OK \vv_2 + \OK \vv_3$ be a classical lattice such that $Q(\vv_1)=1$, $Q(\vv_2)=\lambda$, $Q(\vv_3)=3$. Then $L$ is either isometric to one of the lattices $L_1$, $L_2$, $L_3$, $L_3'$ from \eqref{eq:ListLattices}, or $L\simeq \qf{1,\lambda,2}$ or $L\simeq \qf{1,\lambda,3}$.
	\end{lemma}
	
	Note that the Lemma does \emph{not} assume $\vv_1, \vv_2, \vv_3$ to be linearly independent; but as a corollary, we see that in fact they must be.
	
	\begin{proof}
		Consider the Gram matrix of $\vv_1$, $\vv_2$, $\vv_3$:
		\[
		G = \begin{pmatrix}
			1 & \omega_{12} & \omega_{13}\\
			\omega_{12} & \lambda & \omega_{23}\\
			\omega_{13} & \omega_{23} & 3\\
		\end{pmatrix}.
		\]
		By Cauchy--Schwarz inequality, $\omega_{12}^2\preceq \lambda$, $\omega_{13}^2\preceq 3$ and $\omega_{23}^2 \preceq 3\lambda$. Since $\lambda \preceq 3\lambda$, our assumption \eqref{eq:SmallConditionWeak} on $K$ yields that $\omega_{12},\omega_{13},\omega_{23} \in \Q(\!\sqrt2)$. Thus, Lemma~\ref{le:smallsquares} yields $\omega_{12}=0$, $\omega_{13} \in \{0,\pm1,\pm\sqrt2\}$ and $\omega_{23} \in \{0, \pm1, \pm(1+\sqrt2)\}$.
		
		By checking all possible combinations of $\omega_{13}$ and $\omega_{23}$, we verify the statement. (We can reduce the amount of work if we get rid of the $\pm$ by replacing $\vv_1$ by $-\vv_1$ and $\vv_2$ by $-\vv_2$ if needed; that only leaves $3 \cdot 3$ cases to consider.)
	\end{proof}
	
	By the previous lemma, if $L$ is a classical lattice representing $1$, $\lambda$ and $3$ of type non-LIRE, then (under certain assumptions) it contains either $\qf{1,\lambda,2}$ or $\qf{1,\lambda,3}$. We study these two possibilities separately.
	
	In the following, we will use a somewhat stronger version of condition \eqref{eq:SmallConditionWeak}:
	\begin{equation} \label{eq:SmallCondition}
		\forall\,\omega\in\OK: \quad \omega^2\preceq 3\lambda\ \text{ or }\ \omega^2 \preceq 6 \quad \Longrightarrow \quad \omega \in \Q(\!\sqrt2).
		\tag{{\smaller[1]{\raisebox{-2pt}{\AsteriskRoundedEnds}}}}
	\end{equation}
	
	\begin{lemma} \label{le:diag2}
		Let $K \ni \sqrt2$ be a totally real number field satisfying \eqref{eq:SmallCondition}. Let $L$ be a universal ternary classical lattice over $K$ which is of type non-LIRE. If $L$ contains a sublattice isometric to ${\qf{1,\lambda,2}}$, then $\sqrt{5\lambda}\in K$.
	\end{lemma}
	\begin{proof}
		Consider the Gram matrix corresponding to the three basis vectors of $\qf{1,\lambda,2}$ and the vector representing $3(2-\sqrt2)$:
		\begin{equation}\label{eq:Gram}
			G = \begin{pmatrix}
				1 & 0 & 0 & \alpha \\
				0 & \lambda & 0 & \beta \\
				0 & 0 & 2 & x \\
				\alpha & \beta & x & 3(2-\sqrt2)
			\end{pmatrix}.
		\end{equation}
		By Cauchy--Schwarz inequality, $\alpha^2 \preceq 3(2-\sqrt2)$ and $\beta^2 \preceq \lambda3(2-\sqrt2) = 6$. By assumption, $\beta \in \Q(\!\sqrt2)$. Furthermore, $\alpha^2(1+\sqrt2)^2 \preceq 3(2-\sqrt2)(1+\sqrt2)^2 = 3\lambda$, hence by assumption, $\alpha(1+\sqrt2) \in \Q(\!\sqrt2)$, so we have $\alpha\in\Q(\!\sqrt2)$ as well.
		
		Thus $\alpha \in \{0, \pm1, \pm(1-\sqrt2)\}$ and $\beta \in \{0, \pm1, \pm2, \pm\sqrt2, \pm(1+\sqrt2),\pm(1-\sqrt2) \}$ by parts \ref{le:smallsquares-e} and \ref{le:smallsquares-c} of Lemma~\ref{le:smallsquares}, respectively.
		
		Now, remember that $L$ is a ternary lattice. Thus, $\rank G \leq 3$; in particular, $\det G = 0$. A direct computation yields:
		\[
		\det G = -\lambda x^2 + 2(6-\lambda\alpha^2-\beta^2),
		\]
		so the condition $\det G = 0$ can be rewritten as
		\[
		x^2 = (2-\sqrt2) (6-\lambda\alpha^2-\beta^2).
		\]
		We already know that there are only three possibilities for $\alpha^2$ and six for $\beta^2$. It remains to check all $3\cdot 6$ combinations and compute the corresponding $x$. (Typically, this $x$ will be an algebraic integer of degree $4$.) Since $x^2$ must be totally positive, we can ignore those pairs where $6-\lambda\alpha^2-\beta^2$ is not totally positive. 
		
		First, we handle separately the two cases which yield $x\in\Q(\!\sqrt2)$: If $\alpha=\pm1$, $\beta=\pm\sqrt2$, we get $x^2 = 6-4\sqrt2$, so $x = \pm (2-\sqrt2)$. With this choice of $\alpha, \beta, x$, the Gram matrix is indeed totally positive semidefinite of rank $3$; however, a computation shows that the corresponding vectors generate a lattice isometric to $\qf{1,1,\lambda}\simeq L_1$, which is impossible, as $L$ is of type non-LIRE. Similarly, $\alpha=\pm(1-\sqrt2)$, $\beta=\pm\sqrt2$ yields $x^2 = 2$, so $x = \pm \sqrt2$. This is also impossible, because the vectors again generate a lattice isometric to $L_1$.
		
		The remaining pairs $\alpha^2,\beta^2$ yield that $\det G = 0$ only if 
		\[
		x^2\in\{12-6\sqrt2, 10-5\sqrt2, 4-2\sqrt2, 8-4\sqrt2, 10-7\sqrt2,\lambda,10-6\sqrt2,8-5\sqrt2,6-2\sqrt2, 4-\sqrt2\}.
		\] 
		Note that $12-6\sqrt2 = \sqrt2^23(2-\sqrt2)$ is a square if and only if $3(2-\sqrt2)$ is a square, hence if and only if $3\lambda$ is a square. Applying analogous square reductions for all the numbers (for example $10-7\sqrt2=(1-\sqrt2)^4\lambda$) yields that if $\det G = 0$, then at least one of the following must be a square:
		\[
		3\lambda,\ 5\lambda,\ \lambda,\ 3+\sqrt2,\ 4+\sqrt2,\ 3-\sqrt2,\ 4-\sqrt2.
		\]
		However, most of these numbers cannot be a square in a field which satisfies condition \eqref{eq:SmallCondition}: for example, if $x^2=3\lambda$, then clearly $x^2\preceq3\lambda$, so \eqref{eq:SmallCondition} yields $x\in\Q(\!\sqrt2)$, which is not true. Therefore, we are left with only one possibility, namely that $5\lambda$ is a square in $K$. 
	\end{proof}

	\begin{lemma} \label{le:diag3}
		Let $K \ni \sqrt2$ be a totally real number field satisfying \eqref{eq:SmallCondition}. Let $L$ be a ternary classical lattice over $K$ and suppose that $L$ contains a sublattice isometric to ${\qf{1,\lambda,3}}$. Then $L$ is not universal.  
	\end{lemma}
	\begin{proof}
		The proof is analogous to that of Lemma~\ref{le:diag2}. We consider the Gram matrix
		\[
		G = \begin{pmatrix}
			1 & 0 & 0 & \alpha \\
			0 & \lambda & 0 & \beta \\
			0 & 0 & 3 & x \\
			\alpha & \beta & x & 3(2-\sqrt2)
		\end{pmatrix}.
		\]
		We get the same conditions for $\alpha$ and $\beta$ as before; thus, again, $\alpha \in \{0, \pm1, \pm(1-\sqrt2)\}$ and $\beta \in \{0, \pm1, \pm2, \pm\sqrt2, \pm(1+\sqrt2),\pm(1-\sqrt2) \}$.
		
		This time we get
		\[
		\det G = -\lambda x^2 + 3(6-\lambda\alpha^2-\beta^2),
		\]
		so $\det G = 0$ can be rewritten as
		\[
		x^2 = \frac{3}{\lambda} (6-\lambda\alpha^2-\beta^2).
		\]
		This expression is totally positive precisely for the same pairs $\alpha^2,\beta^2$ as in the previous proof; but this time, we can exclude even more cases, based on the condition $x^2 \in \OK$. This excludes all the cases where the expression on the right side is not an algebraic integer; thus, $\beta^2 \notin \{1^2,(1+\sqrt2)^2,(1-\sqrt2)^2\}$.
		
		The remaining pairs $\alpha^2,\beta^2$ yield that $\det G = 0$ only if
		\[
		x^2\in\{18-9\sqrt2, 6-3\sqrt2, 12-6\sqrt2,15-9\sqrt2,9-6\sqrt2, 9-3\sqrt2, 3\}.
		\] 
		One of these expressions is a square in $K$ if and only one of the following is a square: 
		\[
		\lambda,\ 3\lambda,\ 3(3+\sqrt2),\ 3,\ 3(3-\sqrt2).
		\]
		This time, only the numbers $3(3\pm\sqrt2)$ can be a square in a field fulfilling condition \eqref{eq:SmallCondition}.
		
		Assume that $K\ni\sqrt{\vphantom{\sqrt2}\smash{3(3+\sqrt2)}}$, and consider its subfield $F=\Q\bigl(\!\sqrt{\vphantom{\sqrt2}\smash{3(3+\sqrt2)}}\bigr)$, generated by a root of $\its^4-18\its^2+63$. Looking up this field in the database LMFDB \cite{LMFDB}, we can see that $\abs{\UKPlus[F]/\UKctv[F]}=h^{+}(F)/h(F)=4$. By Lemma~\ref{le:unitsDontDissappear}, we have $\abs{\UKPlus/\UKctv} \geq \abs{\UKPlus[F]/\UKctv[F]}$. Therefore, the lattice $L$ cannot be universal over such field $K$ by Proposition~\ref{pr:KKK}. The same argument applies for $\sqrt{\vphantom{\sqrt2}\smash{3(3-\sqrt2)}}$, since it is also a root of $\its^4-18\its^2+63$.
	\end{proof}
	
	Finally, we put everything together.
	
	\begin{proposition} \label{pr:nonLIRE-5lambda}
		Let $K \ni \sqrt2$ be a totally real number field satisfying \eqref{eq:SmallCondition}. If $K$ admits a universal ternary classical lattice which is of type non-LIRE, then $5\lambda=5(2+\sqrt2)$ is a square in $K$. 
	\end{proposition}
	\begin{proof}
		First, \eqref{eq:SmallCondition} implies the assumptions of Lemma~\ref{le:3x3}. Thus, if $L$ is a universal ternary classical lattice over $K$, then the sublattice generated by the vectors representing $1$, $\lambda$ and $3$ is either isometric to $L_1$, $L_2$, $L_3$ or $L_3'$ (all of which is impossible, as $L$ is of type non-LIRE), or it is isometric to $\qf{1,\lambda,2}$ or $\qf{1,\lambda,3}$.
		
		In the former case, Lemma~\ref{le:diag2} yields that $5\lambda$ is a square; in the latter case, Lemma~\ref{le:diag3} yields that no such universal ternary classical lattice exists. Thus, $5\lambda$ must be a square.
	\end{proof}
	
	Assuming that the universal ternary lattice is free, we can strengthen the result of Proposition~\ref{pr:nonLIRE-5lambda}.
	
	\begin{proposition}\label{pr:nonLIRE-free}
		Let $K \ni \sqrt2$ be a totally real number field satisfying \eqref{eq:SmallCondition}. Then there is no universal ternary classical \textbf{free} lattice over $K$ of type non-LIRE. 
	\end{proposition}
	\begin{proof}
		Suppose that there exists a universal ternary classical lattice $L$ over $K$ of type non-LIRE; $L$ is not necessarily free. For later use, note that $\UKPlus=\UKctv$ by Proposition~\ref{pr:KKK}. By Proposition~\ref{pr:nonLIRE-5lambda}, we must have $\sqrt{5\lambda}\in K$. Denote $F=\Q(\!\sqrt{5\lambda})$. By inspecting the proof of Lemma~\ref{le:diag2}, we see that $L$ contains a sublattice $J$ generated by four vectors with Gram matrix $G$ from \eqref{eq:Gram} corresponding to the values $x^2=5(2-\sqrt2)$, $\beta^2=1$, $\alpha=0$. If we compute how to express the vector representing $2$ as a linear combination of the vectors representing $\lambda$ and $3(2-\sqrt2)$, we arrive at $J\simeq\qf{1}\perp J_0$ where $J_0=\ia\e_1 + \OK \e_2$ and the Gram matrix of $\e_1,\e_2$ is $\bigl(\begin{smallmatrix} 2 & \vartheta \\ \vartheta & 3(2-\sqrt2)\end{smallmatrix}\bigr)$ and $\ia=(\frac{5}{\vartheta},1)$ with $\vartheta=\sqrt{\vphantom{\sqrt2}\smash{5(2-\sqrt2)}}$. (Note that $F$ is a Galois extension of $\Q$, and hence indeed $\vartheta\in F$.)
		
		The lattice $J$ is not free over $F$. It is unimodular. Hence, it does not have any proper overlattice (over any field), so $L=J$. Therefore, the only possibility for a universal ternary classical free lattice of type non-LIRE is this lattice $J$ over a field $K$ over which it becomes free. In particular, $\ia$ must be a principal fractional ideal in $K$. Since $h(F)=2$, it follows that all ideals from $F$ must become principal in $K$. However, in $F$, we have $(2+\sqrt2) = \ip_2^2$. Since $\ip_2$ is a principal ideal in $K$ and $\UKPlus=\UKctv$, we get $2+\sqrt2=\omega^2$; that contradicts the condition \eqref{eq:SmallCondition}.
	\end{proof}
	
	\begin{remark}
		We would like to point out that the result of Proposition~\ref{pr:nonLIRE-5lambda} is quite strong. 
		\begin{enumerate}[wide=0pt]
			\item Note that if $K\ni\sqrt2$ is a field satisfying \eqref{eq:SmallCondition} and $L$ a universal ternary classical lattice over $K$, then $L$ has to contain a sublattice isometric to one of the following lattices:
			\begin{itemize}
				\item $L_1$, $L_2$, $L_3$, or $L_3'$ listed in \eqref{eq:ListLattices} which are defined over $\Q(\!\sqrt2)$; or
				\item the lattice $\qf{1}\perp J_0$ from the proof of Proposition~\ref{pr:nonLIRE-free} which is defined over $\Q(\!\sqrt{5\lambda})$. 
			\end{itemize}
			\item It seems that most fields satisfy the condition \eqref{eq:SmallCondition}; see Remark~\ref{re:AllExcField} for the situation in quartic fields.
		\end{enumerate}
	\end{remark}
	
	\subsubsection{When small elements do exist}
	
	The following fields will need to be handled separately in some of the upcoming proofs. The lower index is the discriminant of the field.
	\[
	\arraycolsep=1.5pt\def\arraystretch{1.8}
	\begin{array}{lclcl}
		K_{1600}&=&\Q\bigl(\!\sqrt2,\sqrt{5}\bigr)=\Q(\alpha) & \hphantom{i} & \text{where $\alpha$ is a root of } \iks^4-6\iks^2+4,\\
		K_{2048}&=&\Q\bigl(\zeta_{16}^{}+\zeta_{16}^{-1}\bigr)=\Q\bigl(\!\sqrt{\lambda}\bigr)=\Q(\alpha) & & \text{where $\alpha$ is a root of } \iks^4-4\iks^2+2,\\
		K_{2624}&=&\Q(\alpha) & & \text{where $\alpha$ is a root of } \iks^4-2\iks^3-3\iks^2+2\iks+1,\\
		K_{10816}&=&\Q\bigl(\!\sqrt2,\sqrt{13}\bigr)=\Q(\alpha) & & \text{where $\alpha$ is a root of } \iks^4-2\iks^3-9\iks^2+10\iks-1,\\
		K_{51200}&=&\Q\bigl(\!\sqrt{5\lambda}\bigr)=\Q(\alpha) & & \text{where $\alpha$ is a root of } \iks^4-20\iks^2+50.\\
	\end{array}
	\]
	As $K_{2624}$ is not a Galois extension of $\Q$, there are two distinct copies contained in $\R$: $\Q\bigl(\!\sqrt{7+2\sqrt2}\bigr)$ and $\Q\bigl(\!\sqrt{7-2\sqrt2}\bigr)$. An interesting fact about $K_{51200}$ is that among totally real quartic fields $K\ni\sqrt2$, it is the first (by discriminant) with $h(K)> 1$.
	
	\bigskip
	
	The following lemma is the only result in this paper which heavily depends on the use of a computer. But the computations are not particularly extensive -- they took around $15$ minutes on a PC.
	
	Remember that by Proposition~\ref{pr:KKK}, the only relevant fields are those where all totally positive units are squares. This is why we assume $\UKPlus=\UKctv$ in the following lemma; we could prove the statement without this assumption (with more exceptional fields) at the cost of running a more extensive computation, see Remark~\ref{re:AllExcField}.

	\begin{lemma} \label{le:solve1308fields}
		Let $K \ni \sqrt2$ be a totally real quartic field with $\UKPlus=\UKctv$. Let $\omega\in \OK$.
		\begin{enumerate}
			\item If $K \neq K_{2048}, K_{2624}$, then $\omega^2 \preceq 3\lambda$ implies $\omega \in \Q(\!\sqrt2)$.
			\item If $K \neq K_{1600}, K_{2048}, K_{2624}, K_{10816}$, then $\omega^2 \preceq 6$ implies $\omega \in \Q(\!\sqrt2)$.
		\end{enumerate}
	\end{lemma}
	\begin{proof}
		We prove both parts at the same time. Assume that $\omega \notin \Q(\!\sqrt2)$ satisfies either of the inequalities.
		
		Assume that $\Q(\omega) = K$. The condition $\omega^2\preceq 6$ is equivalent to $\house{\omega} \leq \sqrt6$, while the condition $\omega^2 \preceq 3\lambda$ implies $\house{\omega} \leq \sqrt{6+3\sqrt2} \approx 3.2004$. In either case, $K = \Q(\omega)$ with $\house{\omega} \leq \sqrt{6+3\sqrt2}$. This yields a bound on the discriminant of $K$. In particular, by Proposition~\ref{pr:KubasBound} we get
		\[
		\disc K \leq \frac{2^{12}}{5^5}\house{\omega}^{12} \leq \frac{2^{12}}{5^5}(6+3\sqrt2)^6 \approx 1\,513\,496.96.
		\]
		The LMFDB database \cite{LMFDB} contains all totally real quartic fields with discriminant up to $10^7$; in particular, it contains all the fields where such an $\omega$ can exist. We downloaded all 1308 totally real quartic fields $K\ni\sqrt2$ with $\disc K < 1\,513\,497$; after disregarding those where $\abs{\UKPlus/\UKctv} > 1$, we were left with 360 fields.
		
		Before we proceed, let us clarify that this list gives fields only up to isomorphism (i.e., up to embedding into $\R$), and these isomorphisms may send $\lambda$ to $2-\sqrt2$. Fortunately, we see that the (in)validity of the implication $\omega^2 \preceq 3\lambda \,\Rightarrow\, \omega\in\Q(\!\sqrt2)$ is stable under field isomorphism, since the isomorphisms map $\Q(\!\sqrt2)$ to itself. Thus, in the process described below, it is indeed enough to check just $360$ fields -- one need not worry about the isomorphisms in the non-Galois case.
		
		In each of these remaining fields, we used a simple program in Magma to find all elements satisfying $\omega^2\preceq 3\lambda$ or $\omega^2\preceq 6$. In almost all cases we got precisely the 5 values listed in Lemma~\ref{le:smallsquares}\ref{le:smallsquares-d} and the 11 values listed in Lemma~\ref{le:smallsquares}\ref{le:smallsquares-c}, all belonging to $\Q(\!\sqrt2)$. The only exceptional fields were $K_{1600}$, $K_{2048}$, $K_{2624}$ and $K_{10816}$, just as the lemma claims.
		
		Finally, let us handle the case when $\Q(\omega) \neq K$. Then $\Q(\omega)$ is a real quadratic field, and $\omega$ must satisfy $2\Tr{\Q(\omega)/\Q}{\omega^2} = \Tr{K/\Q}{\omega^2}\leq 24$. This implies that $\Q(\omega)=\Q(\!\sqrt{n})$ with $n \leq 6$ or $n \equiv 1 \pmod4$, $n \leq 24$, which leaves only very few biquadratic fields $K = \Q(\!\sqrt{2},\sqrt{n})$ to check. Since all of them have rather small discriminants (at most $28\,244$, which corresponds to $\Q(\!\sqrt2,\sqrt{21})$), they were already included in the computation which solved the main part of the proof.
	\end{proof} 
	
	\begin{remark}\label{re:AllExcField}
		Note that the four fields in Lemma~\ref{le:solve1308fields} are the only quartic fields with $\UKPlus=\UKctv$ that violate condition \eqref{eq:SmallCondition}. Out of curiosity, we ran the same program for all $1308$ fields mentioned in the previous proof (no longer excluding those with $\UKPlus \neq \UKctv$), resulting in the following statement.
		
		\textit{
			Let $K\ni \sqrt2$ be a totally real quartic field. Let $\omega \in \OK$.
			\begin{enumerate}
				\item If $K \neq K_{2048}, K_{2624}$ and also $K\neq K_{7168}, K_{18432}$, then $\omega^2 \preceq 3\lambda$ implies $\omega\in\Q(\!\sqrt2)$.
				\item If $K \neq K_{1600}, K_{2048}, K_{2624}, K_{10816}$ and also $K\neq K_{2304}, K_{7168}, K_{14336}$, then $\omega^2 \preceq 6$ implies $\omega \in \Q(\!\sqrt2)$.
			\end{enumerate}
		}
		
		It is interesting to note that these new \enquote{exceptional fields} (which nevertheless require no exceptional treatment in our proof as they contain a nonsquare totally positive unit) are only those fields which we already discovered in proofs of Lemmas~\ref{le:diag2} and~\ref{le:diag3} as fields that clearly violate \eqref{eq:SmallCondition}: $K_{7168}=\Q\bigl(\!\sqrt{3+\sqrt2}\bigr)$, $K_{18432}=\Q\bigl(\!\sqrt{3\lambda}\bigr)$, $K_{2304}=\Q\bigl(\!\sqrt2,\sqrt3\bigr)$ and $K_{14336}=\Q\bigl(\!\sqrt{4+\sqrt2}\bigr)$.
	\end{remark}

	\subsubsection{Lattices of type non-LIRE in quartic fields}

	Finally, we are almost ready to prove that ternary lattices of type non-LIRE are never universal over fields of degree $4$. Note that the solution for a few exceptional fields is postponed -- see Lemma~\ref{le:exceptional}.
	
	\begin{proposition} \label{pr:nonLIRE}
		Let $K \ni \sqrt2$ be a totally real number field with $[K:\Q]=4$. Moreover, suppose that $K\neq K_{1600}, K_{2048}, K_{2624}, K_{10816}, K_{51200}$. If $L$ is a ternary classical lattice over $K$ which is of type non-LIRE, then $L$ is not universal.
	\end{proposition}
	\begin{proof}
		Suppose that $L$ is universal over $K$. First assume that $K$ satisfies \eqref{eq:SmallCondition}. Then $\sqrt{5\lambda}\in K$ by Proposition~\ref{pr:nonLIRE-5lambda}. Since $K$ is quartic, $K$ must be $\Q(\!\sqrt{5\lambda})=K_{51200}$. 
		
		Assume now that there exists an element $\omega \in \OK \setminus \Q(\!\sqrt2)$ satisfying $\omega^2\preceq 3\lambda$ or $\omega^2 \preceq 6$.
		If $\UKPlus=\UKctv$, then by Lemma~\ref{le:solve1308fields}, $K$ is necessarily one of $K_{1600}$, $K_{2048}$, $K_{2624}$ and $K_{10816}$. 
		If $\UKPlus\neq\UKctv$, then there is no universal ternary classical lattice over $K$ by Proposition~\ref{pr:KKK}. 
	\end{proof}
	
	\subsection{Exceptional fields} \label{ss:Exceptional}
	
	It remains to handle the five fields for which the main proof failed for various reasons. Over them, we look at all ternary classical lattices, making no assumptions about the LIRE-type.
	
	\begin{lemma}\label{le:exceptional}
		Over the following quartic fields, there are no universal ternary classical lattices:
		$K_{1600}$, $K_{2048}$, $K_{2624}$, $K_{10816}$, $K_{51200}$.
	\end{lemma}
	\begin{proof}
		The fields $K_{1600}=\Q(\!\sqrt2,\sqrt5)$ and $K_{10816}=\Q(\!\sqrt2,\sqrt{13})$ are biquadratic and of class number $1$; in particular, all lattices over them are free. However, biquadratic fields do not admit any universal ternary classical free lattice by \cite{KTZ}.
		
		For $K=K_{2048}$, we claim that there is no ternary classical lattice which represents all of $1$, $p_2$, $p_{17}$ and $p_{17}'$ where the totally positive elements $p_2$, $p_{17}$ and $p_{17}'$ satisfy $\norm{K/\Q}{p_2}=2$, $\norm{K/\Q} {p_{17}}=\norm{K/\Q}{p_{17}'}=17$ and $p_{17}p_{17}'$ is not a square. (Up to multiplication by squares of units, this defines $p_2$ uniquely, but there are four distinct choices of $p_{17}$ and then three choices for $p_{17}'$.) In fact, it is enough to verify\footnote{Which we do not do here; but it boils down just to checking that $\omega^2 \preceq p_{17}$ or $\omega^2 \preceq p_2p_{17}$ implies $\omega=0$.} that the vectors representing $1$, $p_2$ and $p_{17}$ must be orthogonal. By applying automorphisms (since $K_{2048}$ is a Galois extension), we get that vectors representing $1$, $p_2$ and $p_{17}'$ must also be orthogonal, which then means that the vectors representing $p_{17}$ and $p_{17}'$ differ only by a scalar. This is a contradiction, as $p_{17}p_{17}'$ is not a square.
		
		For $K=K_{2624}$, we claim that no ternary classical lattice represents all of $1$, $\lambda$, $p_7$, $p_7'$, where the totally positive elements $p_{7}$ and $p_{7}'$ satisfy $\norm{K/\Q} {p_{7}}=\norm{K/\Q}{p_{7}'}=7$ and $p_{7}p_{7}'$ is not a square. (Up to multiplication by squares of units and switching $p_7$ with $p_7'$, this defines these elements uniquely.) In fact, again, one can prove that the vectors representing $1$, $\lambda$, $p_7$ must be orthogonal.
		
		For $K=K_{51200}$, we claim that no ternary classical lattice represents all of $1$, $\lambda$, $\alpha_{14}$, $\alpha_{14}'$, where the totally positive elements $\alpha_{14}$ and $\alpha_{14}'$ satisfy $\norm{K/\Q} {\alpha_{14}}=\norm{K/\Q} {\alpha_{14}'}=14$ and $\alpha_{14}\alpha_{14}'$ is not a square. (There are four choices for $\alpha_{14}$ up to multiplication by squares.) In fact, again, one can prove that the vectors representing $1$, $\lambda$, $\alpha_{14}$ must be orthogonal.
	\end{proof}
	
	This finally allows us to prove the main result.
	
	\begin{theorem}\label{th:QuarticKitaokaFreeLattices}
		Let $K$ be a totally real number field with $[K:\Q]=4$. Then there is no universal ternary classical free quadratic lattice over $K$.
	\end{theorem}
	\begin{proof}
		Invoking \cite[Thm.~1.2]{KKK}, we can assume $\sqrt2\in K$. Then we combine Proposition~\ref{pr:noLIRE}, Proposition~\ref{pr:nonLIRE} and Lemma~\ref{le:exceptional}.
	\end{proof}
	
	
	\section{Kitaoka's conjecture over fields where \texorpdfstring{$2$}{2} is not squarefree} \label{se:notSquarefree}
	
	In this short section, we study Kitaoka's conjecture for fields of arbitrary degree under the assumption that $2$ is not squarefree. One possible motivation for focusing on these fields is the fact that by \cite[Thm.~1.1]{KKK}, a universal ternary classical free lattice can only exist if $(2)$ ramifies. We also explicitly exclude fields where $2$ is a square, as they are studied in the rest of the paper. 
	
	Throughout this section, we consider the case when $\UKPlus=\UKctv$. Thus, for every nonzero $t \in \OK$ there exists $\eta_t \in \UK$ such that $\eta_t t \in \OKPlus$. The assumption that $2$ is not squarefree (i.e., it is divisible by some $t^2$ which is not a unit) yields $2=\gamma t^2$, and after possibly replacing $t$ by $\eta_tt$ and adjusting $\gamma$ accordingly, we can assume $\gamma,t \in \OKPlus$.
	
	\begin{lemma} \label{le:squarefree}
		Let $K \not\ni\sqrt2$ be a totally real number field with $\UKPlus=\UKctv$. Let $2=\gamma t^2$ for some $\gamma, t\in \OKPlus$. If there exists a universal ternary classical lattice over $K$, then $t$ or $\gamma t$ is a square.
	\end{lemma}
	\begin{proof}
		For the sake of contradiction, assume that $t\neq\square$ and $\gamma t \neq \square$. Observe that also $\gamma\neq\square$ since $2\neq\square$. Also, $t$ is not a unit, since all totally positive units are squares. Thus all three elements $\gamma$, $t$ and $\gamma t$ have norm strictly below $\norm{K/\Q}{\gamma t^2}=\norm{K/\Q}{2}=2^{[K:\Q]}$ and therefore they are indecomposable by Lemma~\ref{le:smallAreIndecomposable}.
		
		The universal lattice must represent $1$, $t$, $\gamma$ and $\gamma t$. The Gram matrix of the corresponding vectors is
		\[
		G=
		\begin{pmatrix}
			1 & \beta_{12} & \beta_{13} & \beta_{14}\\
			\beta_{12} & t & \beta_{23} & \beta_{24}\\
			\beta_{13} & \beta_{23} & \gamma & \beta_{34}\\
			\beta_{14} & \beta_{24} & \beta_{34} & \gamma t
		\end{pmatrix}
		\]
		for some $\beta_{ij}\in\OK$. Cauchy--Schwarz inequality yields $\beta_{12}^2 \preceq t$; since $t$ is indecomposable and nonsquare, we have $\beta_{12}=0$. For the same reason we get $\beta_{13}=\beta_{14}=\beta_{23}=0$, since $\gamma$ and $\gamma t$ are also nonsquare indecomposables. Further, $\beta_{24}^2 \preceq \gamma t^2 = 2$, so $\beta_{24}=0$ or $\beta_{24}=\pm1$; without loss of generality, we may assume $\beta_{24}=1$ in the latter case. Therefore, we have
		\[
		G=
		\begin{pmatrix}
			1 & 0 & 0 & 0\\
			0 & t & 0 & 0\\
			0 & 0 & \gamma & \beta\\
			0 & 0 & \beta & \gamma t
		\end{pmatrix}
		\qquad
		\text{or}
		\qquad
		G=
		\begin{pmatrix}
			1 & 0 & 0 & 0\\
			0 & t & 0 & 1\\
			0 & 0 & \gamma & \beta\\
			0 & 1 & \beta & \gamma t
		\end{pmatrix}.
		\]
		Since the lattice is ternary, $\det G=0$. In the former case, this yields $\gamma^2 t = \beta^2$, so $t$ is a square (in $K$ and thus also in $\OK$), a contradiction. In the latter case, one gets $\gamma^2t^2-\gamma-\beta^2t=0$, which can be rewritten as $2\gamma=\gamma+\beta^2t$, yielding $\gamma= \beta^2 t$, so $\gamma t$ is a square (also a contradiction).
	\end{proof}
	
	Note that the assumption $\UKPlus=\UKctv$ can be replaced by $t\notin\UK$, since it is only used at the beginning of the proof. However, as mentioned before, the condition $\UKPlus=\UKctv$ is necessary in all the applications below, because it allows us to assume $t \in\OKPlus$.
	
	The above lemma seems to be applicable in most fields not containing $\sqrt2$ where $2$ is not squarefree (see in Proposition~\ref{pr:2notSquarefreeDeg6} below for the strong corollary it yields for fields of degree $6$); we primarily use it to fully solve this situation for fields of degree $2^k$.
	
	\begin{proposition} \label{pr:2notSquarefreeDeg2power}
		Let $K$ be a totally real number field of degree $2^k, k\in\Z^+$, with $\UKPlus=\UKctv$. Assume that $\sqrt2\notin K$ and that $2$ is not squarefree in $K$. Then there is no universal ternary classical lattice over $K$.
	\end{proposition}
	
	A significant part of the proof is contained in the following lemma.
	
	\begin{lemma}\label{le:claim}
		Let $K\not\ni\sqrt2$ be a totally real number field of degree $d$ with $\UKPlus=\UKctv$. Assume that $K$ admits a universal ternary classical lattice. If $2$ can be written as $2=\gamma t^2$ for $\gamma,t \in \OKPlus$ with $\norm{K/\Q}{t}=2^j$, then it can also be written as $2=\tilde\gamma \tilde{t}^2$ for $\tilde\gamma,\tilde{t} \in \OKPlus$ with $\norm{K/\Q}{\tilde{t}} \in \{2^{j/2},2^{(d-j)/2}\}$.
		
		In particular, the exponent $j$ cannot be odd.
	\end{lemma}
	\begin{proof}
		By Lemma~\ref{le:squarefree}, we know that $t$ or $\gamma t$ is a square. In the former case, write $t=\omega^2$. Let $\eta\in\UK$ such that $\eta\omega\in\OKPlus$. Then $2 = \gamma \omega^4 = \gamma\omega^2\eta^{-2} (\eta\omega)^2$. If we put $\tilde{t}=\eta\omega$, clearly $\norm{K/\Q}{\tilde{t}} = \sqrt{\norm{K/\Q}{t}} = 2^{j/2}$.
		
		In the latter case, write $\gamma t=\omega^2$ and let again $\eta\in\UK$ such that $\eta\omega\in\OKPlus$. Then $2 = t \omega^2 = t\eta^{-2} (\eta\omega)^2$. If we put $\tilde{t}=\eta\omega$, we get $\norm{K/\Q}{\tilde{t}} = \sqrt{\norm{K/\Q}{\gamma t}} = \sqrt{\norm{K/\Q}{2/t}} = 2^{(d-j)/2}$.
		
		Finally, by \cite{EK}, the existence of a universal ternary lattice implies that the degree $d$ is even. Thus, if $j$ were odd, then neither $j/2$ nor $(d-j)/2$ would be an integer, which is a contradiction. 
	\end{proof}
	
	Now, to prove the proposition, we repeatedly apply this lemma in a field of degree $2^k$.
	
	\begin{proof}[Proof of Proposition~\ref{pr:2notSquarefreeDeg2power}]
		By assumption, $2=\gamma t^2$ for $\gamma\in\OKPlus$ and $t\in\OK\setminus \UK$. Note that $\gamma$ is also not a unit, since otherwise $2$ would be a square due to $\UKPlus=\UKctv$. After possibly replacing $t$ by $\eta_t t$ and $\gamma$ by $\gamma \eta_t^{-2}$, we can assume $t\in\OKPlus$.
		
		For the sake of contradiction, assume that there exists a universal ternary classical lattice over $K$. We have $2=\gamma t^2$ for $\gamma,t\in\OKPlus$ where $\norm{K/\Q}{t} = 2^j$ with $0<j<d/2$. (The values $j=0$ and $j=d/2$ are impossible, as neither $t$ nor $\gamma$ are units.) This inequality is impossible if $d=2^1$. Thus $k\geq 2$, and Lemma \ref{le:claim} will now allow us to use infinite descent to get a contradiction.
		
		First observe that if $j$ is odd, then Lemma \ref{le:claim} yields an immediate contradiction. On the other hand, we will show that if $t$ is replaced by $\tilde{t}$, then the $2$-valuation of $j$ decreases. Indeed: Since $j < d/2=2^{k-1}$, we have $v_2(j)<v_2(d/2)$ and thus $v_2(d/2 - j/2) = v_2(j/2) = v_2(j) - 1$. This allows us to decrease $v_2(j)$ step by step, ultimately getting an odd $j$ and thus a contradiction.
	\end{proof}
	
	We also show the corollary of Lemma~\ref{le:squarefree} for fields of degree $6$:
	
	\begin{proposition} \label{pr:2notSquarefreeDeg6}
		Let $K$ be a totally real number field of degree $6$ with $\UKPlus=\UKctv$. Assume that $\sqrt2\notin K$ and that $2$ is not squarefree in $K$. If $K$ admits a universal ternary classical lattice, then $2=\eta\alpha^3$ for some $\eta\in\UK$, $\alpha\in\OK$. 
	\end{proposition}
	\begin{proof}
		We can again assume that $2=\gamma t^2$ for $\gamma, t \in \OKPlus$, and since neither $\gamma$ nor $t$ is a unit, $\norm{K/\Q}{t} = 2^1$ or $2^2$. The former case is impossible, since by Lemma \ref{le:claim}, the exponent is even. 
		
		Thus $2=\gamma t^2$ with $\norm{K/\Q}{\gamma}=\norm{K/\Q}{t} = 2^2$. By Lemma~\ref{le:squarefree}, $t$ or $\gamma t$ is a square. We shall now prove that $t\neq \square$. Indeed, if $t=\omega^2$, it would allow us to write $2=\gamma\omega^4=\gamma\omega^2\eta_{\omega}^{-2}(\eta_{\omega}\omega)^2=\tilde{\gamma}\tilde{t}^2$ with $\norm{K/\Q}{\tilde{t}} = 2^1$; and we already proved that this is impossible.
		
		Thus $t\neq\square$, so the only remaining case is $\gamma t = \omega^2$. We now aim to prove that $(\gamma)=(t)$ as ideals. Assume not. Then, due to their norm and the fact that the ideal $(\gamma t)$ is a square, the only possibility is $(\gamma) = \ip^2$ and $(t)=\mathfrak{q}^2$ for distinct prime ideals $\ip$, $\mathfrak{q}$. This gives (after possibly replacing $\omega$ by $\eta\omega$ as usual) that $2=t\omega^2$ for $(t)=\mathfrak{q}^2$ and $(\omega)=\ip\mathfrak{q}$. Clearly, neither $\omega$ nor $t\omega$ are squares; this contradicts Lemma~\ref{le:squarefree}.
		
		Therefore $(\gamma)=(t)$ and $(2)=(\gamma)^3$. This proves the claim.
	\end{proof}
	
	Using the same technique for fields of degree $10$, one can prove $2=\eta\alpha^5$. This suggests that the following might be obtainable as a corollary of Lemma~\ref{le:squarefree}: \emph{Let $K$ be a totally real number field with $\UKPlus=\UKctv$ such that $2$ is not squarefree. If $K$ admits a universal ternary classical lattice, then $2=\eta\alpha^k$ for some $\eta\in\UK$, $\alpha\in\OK$, $k\geq2$.} Of course, if $k$ is even, this just recovers the case $\sqrt2\in K$ which was explicitly excluded in our formulations so far; if we forbid $2=\square$, then $k\geq 3$ will be an odd divisor of $[K:\Q]$.

	
	\section{Notes on Explicit Kitaoka's Conjecture} \label{se:NotesConj}
	
	In the introduction, we proposed Conjecture~\ref{co:KitaokaStrong}, namely that the only fields admitting a universal ternary classical lattice are $\Q(\!\sqrt{n})$ for $n=2,3,5$. This is of course supported by Theorem~\ref{th:QuarticKitaoka2}; but let us now look at some evidence that this conjecture holds for all lattices, not just for free ones.
	
	For quadratic fields, the full, lattice version of this conjecture has been proven only recently in \cite[Thm.~7.3]{KK+}. For quartic fields, we have the following partial result.
	
	\begin{proposition} \label{pr:quartic250k}
		Let $K$ be a totally real number field of degree $4$ with $\disc K \leq 250\,000$. Then there is no universal ternary classical lattice over $K$.
	\end{proposition}
	
	In other words, this stronger version of Kitaoka's conjecture holds for the $2672$ quartic fields with the smallest discriminant. Since the exceptional fields tend to be those with a small discriminant (among real quadratic fields, a universal ternary classical lattice exists exactly for the three with the least discriminant), this indicates the validity of this conjecture for all quartic fields.
	
	\begin{proof}[Proof of Proposition~\ref{pr:quartic250k}]
		First, note that if $h(K)=1$, then there is no universal ternary classical lattice over $K$ by Theorem~\ref{th:QuarticKitaoka2}. Second, the results of \cite{KKK} show that a quartic field with a universal ternary lattice must have $\UKPlus=\UKctv$. The reason is as follows: $\abs{\UKPlus/\UKctv} \geq 4$ is clearly impossible (\cite[Prop.~2.5]{KKK}); and if $\abs{\UKPlus/\UKctv} =2$, then $2\OKPlus \subset \sum\square$ by \cite[Thm.~3.1]{KKK}, but by \cite{KY-EvenBetter}, this only allows $\Q(\!\sqrt2,\sqrt5)$ and $\Q(\zeta_{20}+\zeta_{20}^{-1})$ which have class number $1$.
		
		So the only fields to be handled are those where $h(K)>1$ and $\UKPlus=\UKctv$. There are only $36$ such quartic fields with $\disc K \leq 250\,000$, and we have tested them all with a simple script in Magma and found that they do not admit a universal ternary lattice.
		
		In fact, in all $36$ cases we found four elements $\alpha_1,\alpha_2,\alpha_3,\gamma\in\OKPlus$ with the following properties: 
		\begin{enumerate}[(1)]
			\item If $Q(\vv_i)=\alpha_i$ in a classical lattice, then $\vv_1,\vv_2,\vv_3$ are orthogonal vectors. \label{en:1}
			\item $\gamma$ is not represented by the dual of $\qf{\alpha_1,\alpha_2,\alpha_3}$.
		\end{enumerate} 
		Together, this implies that these four elements cannot be simultaneously represented by a ternary classical lattice. Both of these properties are easily checked on a computer; in particular, property~\ref{en:1} is equivalent to checking that $\omega^2 \preceq \alpha_i\alpha_j$ implies $\omega=0$.
	\end{proof}
	
	\medskip
	
	Let us move from quartic fields to fields of higher degree. The main results supporting Kitaoka's conjecture are contained in \cite{EK,KY-weak,KKK} as we already discussed. Some conclusions of the present paper hold for fields of arbitrary degree, in particular Propositions~\ref{pr:nonLIRE-5lambda} and~\ref{pr:nonLIRE-free} for fields containing $\sqrt2$ and all of Section~\ref{se:notSquarefree} for fields where $2$ is not squarefree.
	
	Also, recall that Kala--Yatsyna \cite{KY-EvenBetter} studied Kitaoka's conjecture for fields of the form $\Q(\zeta_k^{}+\zeta_k^{-1})$ where $k$ is a prime or a power of $2$. Let us conclude by announcing that, drawing on the results both from \cite{KKK} and the present paper, we fully solved this case -- in a separate article, we shall prove the following: \emph{The field $\Q(\zeta_k^{}+\zeta_k^{-1})$ admits a universal ternary classical quadratic form if and only if it is $\Q(\!\sqrt2)$, $\Q(\!\sqrt3)$ or $\Q(\!\sqrt5)$.} Since all the known exceptional fields belong to this class, we consider it quite a convincing piece of evidence towards Conjecture~\ref{co:KitaokaStrong}.

\end{document}